\documentclass[hidelinks,onefignum,onetabnum]{siamart251216}

\usepackage{graphicx,color}
\usepackage{amssymb,amsfonts}
\usepackage{mathrsfs}
\usepackage{subfigure}
\usepackage{url}
\usepackage{booktabs}
\usepackage{array}
\usepackage{xcolor}
\usepackage{mathtools} 		
\usepackage{upgreek}
\usepackage{dsfont}
\usepackage{calc}
\usepackage[shortlabels]{enumitem}
\definecolor{lightblue}{RGB}{90,170,255}
\definecolor{myblue}{RGB}{0,90,160}
\usepackage{textcomp}
\usepackage{algorithm}
\usepackage{algpseudocode}
\usepackage{aliascnt}
\usepackage{makecell} 

\graphicspath{{./}{./images/}}

\def\BibTeX{{\rm B\kern-.05em{\sc i\kern-.025em b}\kern-.08em
    T\kern-.1667em\lower.7ex\hbox{E}\kern-.125emX}}

\headers{Direct Search Methods for Online Nonconvex Optimization}{G. Robert and G. Bianchin}

\newsiamthm{problem}{Problem}
\newsiamremark{remark}{Remark}
\newsiamremark{assumption}{Assumption}
\newsiamremark{example}{Example}

\crefname{theorem}{Theorem}{Theorems}
\Crefname{theorem}{Theorem}{Theorems}
\crefname{lemma}{Lemma}{Lemmas}
\Crefname{lemma}{Lemma}{Lemmas}
\crefname{definition}{Definition}{Definitions}
\Crefname{definition}{Definition}{Definitions}
\crefname{corollary}{Corollary}{Corollaries}
\Crefname{corollary}{Corollary}{Corollaries}
\crefname{proposition}{Proposition}{Propositions}
\Crefname{proposition}{Proposition}{Propositions}
\crefname{problem}{Problem}{Problems}
\Crefname{problem}{Problem}{Problems}
\crefname{remark}{Remark}{Remarks}
\Crefname{remark}{Remark}{Remarks}
\crefname{assumption}{Assumption}{Assumptions}
\Crefname{assumption}{Assumption}{Assumptions}
\crefname{example}{Example}{Examples}
\Crefname{example}{Example}{Examples}
\crefname{equation}{equation}{equations}
\Crefname{equation}{Equation}{Equations}
\crefname{figure}{Fig.}{Figs.}
\Crefname{figure}{Fig.}{Figs.}
\crefname{table}{Table}{Tables}
\Crefname{table}{Table}{Tables}
\crefname{algorithm}{Algorithm}{Algorithms}
\Crefname{algorithm}{Algorithm}{Algorithms}
\crefname{section}{Section}{Sections}
\Crefname{section}{Section}{Sections}

\newcommand{\mc}{\mathcal}

\newcommand{\real}{\mathbb{R}}

\newcommand{\naturalpos}{\mathbb{N}_{>0}}
\newcommand{\naturalnneg}{\mathbb{N}_{\geq 0}}

\newcommand*{\QEDB}{\hfill\ensuremath{\square}}

\newcommand*{\QEDBL}{\hfill\ensuremath{\blacksquare}}

\usepackage{pifont} 
\renewcommand{\ast}{\text{\ding{86}}}

\newcommand{\evenset}[1]{[#1]_{\mathrm{2}}}
\newcommand{\Neventwo}{\mathbb{N}^{\mathrm{even}}_{\ge 2}}
\newcommand{\Nevennneg}{\naturalnneg^{\mathrm{even}}}

\DeclareMathAlphabet{\mymathbb}{U}{BOONDOX-ds}{m}{n}

\usepackage{colonequals}

\def\defeq{\colonequals}

\newcommand{\gbmargin}[1]{}
\newcommand{\blue}[1]{#1}

\newcounter{cond}

\title{Direct Search Methods for Online Nonconvex Optimization under Inexact Bandit Feedback%
\thanks{This work was supported in part by the FRFS WEL-T Investigator Programme under Grant WELT-X.8004.25 and in part by the ARC under Grant 24/29-140.}}

\author{Gaspar Robert\thanks{ICTEAM Institute, Universit\'e 
catholique de Louvain,
Belgium (\email{gaspar.robert@uclouvain.be}).}
\and Gianluca Bianchin\thanks{ICTEAM Institute, Universit\'e catholique de Louvain, WEL Research Institute,
Belgium (\email{gianluca.bianchin@uclouvain.be}).}}

\hypersetup{
  pdftitle    = {Direct Search Methods for Online Nonconvex Optimization under
                 Inexact Bandit Feedback},
  pdfauthor   = {Gaspar Robert and Gianluca Bianchin},
  pdfsubject  = {Time-varying optimization under inexact zeroth-order (bandit)
                 feedback: a randomized two-point direct-search algorithm with
                 iteration-complexity guarantees, and its application to online
                 optimal equilibrium selection in control systems},
  pdfkeywords = {bandit optimization; time-varying optimization; direct-search
                 methods; online nonconvex optimization; feedback optimization;
                 zeroth-order optimization; 90C26; 90C56; 90C60; 93C40},
  pdfcreator  = {LaTeX with hyperref},
}

\begin{document}

\maketitle

\begin{abstract}
Optimization under zeroth-order (i.e., bandit) feedback is central to many 
engineering problems where the analytic forms of objectives and/or 
constraints are unavailable. In modern applications, such as online control 
and online learning, optimization problems often evolve with time, requiring 
adaptive optimization methodologies. 
Yet, existing methods in this setting are largely confined to 
adaptations of methodologies developed for time-invariant or first-order 
optimization, and thus often rely on gradient surrogates that fail to fully 
exploit the zeroth-order structure of the available information.
In this paper, we propose a randomized two-point direct-search algorithm for nonconvex time-varying optimization and derive iteration-complexity bounds under both constant and diminishing probing ratios. 
The resulting analysis yields explicit stationarity bounds in terms of the 
temporal variability of the problem and possible oracle errors. 
Our complexity bounds recover the complexity of existing zeroth-order methods 
in the time-invariant setting, while extending direct-search methods beyond static settings. 
As an illustrative application, we show that the methodology is naturally suited to solve optimal (equilibrium-selection) control problems for  dynamical systems. In this setting, the analysis yields explicit stationarity bounds in terms of the temporal variability of the problem, measured through the effects of plant dynamics and exogenous disturbance variations. 
\end{abstract}



\gbmargin{Things to do in the SIOPT file: update intro; update caption of the table; search and remove $\mathcal U$}
 
\section{Introduction}
\label{sec:intro}

\subsection{Motivation}
In many modern applications, optimization problems are 
inherently dynamic: objectives and constraints evolve as new data arrives or 
as the environment changes. Examples arise in machine 
learning~\cite{YC-YZ:20,AR-KS:13}, signal processing~\cite{FJ-AR:13}, 
robotics~\cite{AD-VC-AG-GR-FB:23}, and industrial control~\cite{CB-BS-BD:09}. 
In such settings, repeatedly solving each problem to optimality is 
computationally impractical. This perspective motivates online (time-varying) 
optimization, where decisions are refined incrementally as the problem 
evolves. When successive instances exhibit regularity (e.g., bounded 
variation), past information can be leveraged 
to accelerate updates, enabling lightweight methods that perform only a few 
operations per step, or often only a single iteration.

In this work, we seek online solutions to the sequence of optimization problems
\begin{align}\label{eq:optimization_main}
    \min_{u\in\real^p}\ \Phi_t(u),
    \qquad t\in\naturalnneg, 
\end{align}
where $u \in \real^p$, and $\Phi_t: \real^p \rightarrow \real$ for each $t.$
We are motivated by applications where the objective is accessible only 
through inexact, zeroth-order oracle evaluations:
\begin{align}\label{eq:oracle_inexact}
(t, u) \mapsto \tilde{\Phi}_t(u),
\end{align}
where $\tilde{\Phi}_t:\real^p \to \real$ models a proxy for $\Phi_t(\cdot)$.
The precise assumptions on the objective sequence and the inexact oracle are
stated in \Cref{sec:prob_formulation}.
This setting arises, for instance, in control systems, robotics, and
industrial process optimization, where the objective depends on the response
of an unknown dynamical system and can only be evaluated through online
measurements. It is also relevant in machine learning and signal processing,
where the relationship between the decision variables and the loss may be
unknown, and gradient information may be unavailable or prohibitively
expensive to compute. Another representative example is portfolio
optimization, where an investor only observes the realized return associated
with the selected portfolio, rather than the entire return function.
It is worth noting that \eqref{eq:oracle_tv_problem} corresponds to the
\emph{bandit information} setting \cite{AF-AK-HM:04}, in which only the reward 
\emph{value} associated with the selected decision, namely
$\widetilde{\Phi}_t(u_t)$, is observed. This contrasts with the
\emph{full-information} setting, where the entire reward \textit{function}
$\Phi_t(\cdot)$ is revealed to the agent after each~round.

A central limitation of existing approaches in this setting is that they 
typically guarantee convergence only to an approximate stationary 
point~\cite{YN-VS:17}. This limitation stems from algorithmic strategies that 
seek to emulate first-order methods using noisy gradient estimates, obtained, 
for example, through finite differences or random perturbations, when exact 
gradient information is unavailable~\cite{IS-DS-JM:19}.
In this work, we depart from this paradigm and propose the first direct-search 
method for time-varying optimization under bandit feedback. Rather than 
approximating a higher-order oracle, the proposed method fully exploits the 
zeroth-order nature of the available information. Specifically, it probes the 
loss along random directions and updates the iterates through an accept--reject 
rule.

While inspired by recent advances in zeroth-order optimization for static 
problems~\cite{EB-EG-PR:20}, our approach addresses challenges that do not 
arise in the static setting. In static optimization, several oracle evaluations 
may be performed while the objective remains unchanged. By contrast, in 
time-varying problems, only a single oracle query is available before the 
objective evolves. We therefore develop a direct-search scheme that operates 
under a two-query-per-step budget. Moreover, motivated by the applications, our 
analysis explicitly accounts for inaccuracies in the zeroth-order oracle. As a 
representative application, we consider online optimal equilibrium selection in 
control systems and show that it naturally fits within the proposed framework.

\subsection{Related works} 
Online Convex Optimization (OCO)  was launched by the seminal 
work~\cite{MZ:03}, which proposed an online gradient descent method for 
unconstrained OCO and established an $O(\sqrt{T})$ regret bound. 
This result was later improved in~\cite{EH-AA-SK:07}. While these early 
works considered static offline benchmarks, subsequent contributions extended 
these to dynamic environments with time-varying comparators, establishing 
sublinear dynamic regret guarantees~\cite{OB-YG-AZ:15,EH-RW:15,AM-SS-AJ-AR:16}. 
Relative to this body of literature, our focus is the nonconvex setting. 

Most of the OCO literature relies on first-order information. Bandit OCO was 
initiated by the work of Flaxman et al.~\cite{AF-AK-HM:04}, where a one-point gradient 
estimator was proposed and an $O(T^{3/4})$ regret bound was established. 
Later works developed improved algorithms for multi-point feedback settings and 
obtained tighter regret bounds~\cite{AA-OD-LX:10}. More recently, bandit OCO has
also been studied using two-point gradient estimators~\cite{IS-DS-JM:19}, 
accounting for switching costs, predictions, and dynamic environments
\cite{NCA-AW-SB-LLA:15,NCA-JC-ZL-AG-AW:16}. Relative to these techniques, our 
approach does not seek to mimic a first-order method (relying on gradient 
surrogates), but rather seeks a direct-search methodology. 

Concerning Online Non-Convex Optimization (ONCO), early contributions 
\cite{EH-SK:12} established an $O(\sqrt{T})$ regret guarantee for the
special case of submodular losses. A different line of work was pursued in
\cite{LZ-TY-RJ-ZZ:15}, which considered a structured class of non-convex
bandit problems where each loss is obtained by composing a non-increasing
scalar function with a slowly varying linear mapping. The proposed method
achieves a regret of $\tilde O(\mathrm{poly}(p)\,T^{2/3})$, with $p$ denoting
the dimension of the decision space. The authors of \cite{WK-MB-CT-AB:15,OAM-RM:10,LY-LD-MH-CT-WW:18} leverage exponential-weighting techniques to 
obtain regret bounds of order $O(\sqrt{T\log T})$.
Constrained variants have been considered in \cite{DS-IS-JM-MZ:18,YY-ED-AB-SL:22}. 
Most existing works in ONCO adopt regret as the primary performance metric, as 
convergence to a global minimizer is generally unattainable~\cite{EH-KS-CZ:17}. 
In this work, we instead focus on stationarity guarantees, which provide a more 
general characterization of optimization performance. As a byproduct, we show in \Cref{sec:regret_analysis} that the proposed 
framework also admits a local-regret interpretation and yields regret bounds 
comparable to those established in the ONCO literature.

Our approach builds on direct-search methods for static optimization
\cite{TGK-RML-VT:03,ARC-KS-LNV:09,LNV:13,MOD-LNV:16}, including randomized
variants with improved complexity guarantees for smooth nonconvex
problems~\cite{SG-CWR-LNV-ZZ:15,EB-EG-PR:20}. Despite this progress, the
existing theory remains confined to static objectives and exact function
evaluations, assumptions that exclude many online applications. Extending
direct search beyond this regime is nontrivial: the objective may change
between consecutive queries, while oracle inaccuracies can mask genuine
descent. This work overcomes these limitations by developing the first
direct-search framework for time-varying optimization with inexact bandit
feedback, together with guarantees that explicitly quantify the effects of
temporal drift, oracle error, and a limited query budget.

Finally, our work is motivated by control applications and connects with the 
feedback optimization literature \cite{GB-JC-JP-ED:21-tcns,AH-SB-GH-FD:20}. 
Particularly relevant are the descent-based approach of \cite{AM-GB:25b}, which 
emulates a first-order method, and the direct-search approach of \cite{GG:26}, 
which instead assumes time-scale separation between plant and optimization.

\subsection{Statement of contributions}
The main contributions are threefold. 
First, we propose a direct-search algorithm for nonconvex time-varying 
optimization under inexact zeroth-order information. Unlike static  direct-search 
methods, the proposed approach uses only two function evaluations per step; unlike 
existing descent-based methods, it does not rely on gradient surrogates.
Second, we establish non-asymptotic stationarity guarantees explicitly accounting 
for temporal drifts of the loss and oracle inexactness, and derive 
iteration-complexity bounds under both constant and diminishing probing ratios. 
Since global optimality is generally unattainable in nonconvex 
optimization, even in offline settings \cite{EH-KS-CZ:17}, our analysis focuses 
on stationarity (providing guarantees on the expected gradient norm), and recovers 
regret bounds as a special case. Third, we show how the framework naturally 
applies to the control of dynamical systems, yielding convergence guarantees that quantify 
the interplay between optimization progress, temporal drift, and plant-induced 
oracle errors.

\subsection{Comparison with related methods}
\label{sec:comparison_related}

\Cref{tab:comparison_tv} positions the guarantees established in this paper
against representative first-order and zeroth-order methods. Two points are
worth emphasizing. First, in the absence of temporal drift and oracle
inexactness ($D_T=B_T=0$), Rows~6--7 reduce to Rows~3--4: the proposed method
recovers the dimension and accuracy scalings of static direct search, and in
particular the classical zeroth-order iteration complexity
$\mathcal{O}(p\varepsilon^{-2})$. Second, the last column exhibits the trade-off
induced by the probing schedule: the constant-probing regime accommodates the
same linear drift growth $D_T=\mathcal{O}(T)$ as the local-regret framework
of~\cite{EH-KS-CZ:17} (Row~5), whereas the diminishing-probing regime attains
the sharper $\mathcal{O}(\sqrt{T}\log T)$ scaling, at the price of requiring
$D_T,B_T=\mathcal{O}(\sqrt{T})$.

\begin{table*}[t]
\centering
\setlength{\tabcolsep}{1pt}
\renewcommand{\arraystretch}{1.25}
\resizebox{\textwidth}{!}{%
\begin{tabular}{clccccc}
\toprule
& \textbf{Method} 
& \textbf{Oracle} 
& \textbf{Admissible drift} 
& \textbf{Complexity} 
& \makecell{\textbf{Stationarity} \\  
$\min_t \mathbb{E}\|\nabla \Phi_t(u_t)\|$}\\
\midrule

\multicolumn{6}{c}{\textbf{Time-invariant}} \\
\midrule

1
& Gradient Descent~\cite{YN:18}
& 1st 
& $D_T = 0$, $B_T=0 $
& $\mathcal{O}(\varepsilon^{-2})$
& $\mathcal{O}(\frac1{\sqrt{T}})$\\

2
& 
Descent-based~\cite{YN-VS:17}
& 0th
& $D_T = 0$, $B_T=0 $
& $\mathcal{O}(p\varepsilon^{-2})$
& $\mathcal{O}(\frac1{\sqrt{T}})+\mathcal{O}(p \delta )$\\

3
& 
Direct search STP~\cite{EB-EG-PR:20} (constant $\delta$)
& 0th 
& $D_T = 0$, $B_T=0 $
& $\mathcal{O}(p\varepsilon^{-2})$
& $\mathcal{O}\left(\sqrt{p}(\frac{1}{\delta T} + \delta)
\right)$\\

4
& 
Direct search STP~\cite{EB-EG-PR:20}  ($\delta_t=\frac{1}{\sqrt{t+1}}$)
& 0th 
& $D_T = 0$, $B_T=0 $
& $\mathcal{O}\!\left(
p\varepsilon^{-2}
\log^2\!\left(\frac{\sqrt p}{\varepsilon}\right)
\right)$ & --\\
\midrule

\multicolumn{6}{c}{\textbf{Time-varying}} \\
\midrule

5
& 
Sliding-window gradient descent~\cite{EH-KS-CZ:17}
& 1st 
& $D_T = \mc O(T)$
& -- 
& $\mathcal{O}(L_{\Phi} \sqrt{T})$\\[7pt]

6
& 
\textbf{This work (constant $\delta$)}
& 0th
& \makecell{$D_T = \mc O(T)$, \\$B_T=\mc O(T)$}
& $\mathcal{O}(p\varepsilon^{-2})$
& $\mathcal{O}\left(\sqrt{p}(\frac{1}{\delta T} + \delta
+
\frac{D_T+B_T}{\delta T})
\right)$\\[7pt]

7
& 
\textbf{This work ($\delta_t=\frac{1}{\sqrt{t+1}}$)}
& 0th
& \makecell{$D_T = \mc O(\sqrt{T})$,\\ $B_T=\mc O(\sqrt{T})$}
& $\mathcal{O}\!\left(
p\varepsilon^{-2}
\log^2\!\left(\frac{\sqrt p}{\varepsilon}\right)
\right)$
& 
$\mathcal{O}\!\left(\sqrt{p}\,\frac{\log T}{\sqrt{T}}\right)
+
\mathcal{O}\!\left(\sqrt{p}\,\frac{D_T+B_T}{\sqrt{T}}\right)$\\

\bottomrule
\end{tabular}%
}
\caption{Comparison of stationarity guarantees for first-order and zeroth-order
methods. Here, \(p\) denotes the decision dimension, \(T\) the iteration
horizon, \(\delta_t\) the probing radius, \(D_T\) the cumulative temporal drift
defined in \eqref{eq:objective_drift}, and \(B_T\) the cumulative oracle
inexactness defined in \eqref{as:inexact_oracle}. The results reported in Rows~6--7 follow from
\eqref{eq:main-cor-const-tv} together with
\eqref{eq:constant_delta}, and from
\eqref{eq:split_terms}, respectively.
Rows~6--7 show that our method recovers the same dimension and accuracy scalings 
in the absence of drift and oracle errors (see Rows~3--4), while extending the 
analysis to time-varying objectives through the explicit dependence on
\(D_T\) and \(B_T\). Moreover, a comparison between Row~6 and Row~2--3 shows that 
our method retains the classical zeroth-order iteration complexity 
\(\mathcal O(p\varepsilon^{-2})\), matching both the static STP method (Row~3) 
and descent-based zeroth-order schemes (Row~2). This also recovers the well-known dimension-dependent penalty incurred by randomized zeroth-order 
methods relative to deterministic first-order methods (Row~1) \cite{YN-VS:17}.
The constant-probing regime (Row~6) achieves the same
\(\mathcal O(T)\) scaling while accommodating the same linear drift growth
\(D_T=\mathcal O(T)\). In contrast, the diminishing-probing regime (Row~7)
yields the improved bound \(\mathcal O(\sqrt{T}\log T)\), at the price of
requiring slower growth of the cumulative drift and oracle inexactness.
}
\label{tab:comparison_tv}
\end{table*}

\subsection{Paper organization and notation}

\Cref{sec:prob_formulation} formulates the time-varying optimization problem and introduces the oracle models. \Cref{sec:tv_algorithms} presents the proposed direct-search algorithm and its convergence analysis. \Cref{sec:feedback_opt_description} specializes the framework to feedback optimization and online equilibrium selection in dynamical systems. 
\Cref{sec:extensions} discusses heuristic extensions. 
Numerical results are reported in \Cref{sec:simulation_results}, and concluding remarks are given in \Cref{sec:conclusions}.

We denote by $\naturalnneg \defeq \{0,1,2,\dots\}$ and 
$\naturalpos \defeq \{1,2,\dots\}$ the sets of nonnegative and positive 
integers, respectively. 
We further denote by $\Nevennneg \defeq \{2k \mid k \in \naturalnneg\}$ and 
$\Neventwo \defeq \{2k \mid k \in \naturalpos\}$ the sets of nonnegative even integers and even integers greater than or equal to $2$, respectively. 
For any $T \in \Neventwo$, we define $\evenset{T} \defeq \{0,2,\dots,T-2\}$.
For $u \in \real,$ we denote by $(u)_+ := \max\{u,0\}$.
Given a set $A$, we denote by $\mathbf{1}_{A}$ its indicator function, 
defined as $\mathbf{1}_{A}=1$ if the event $A$ occurs, and 
$\mathbf{1}_{A}=0$ otherwise. 
A function $f(T)$ is said to be $O(T^{-\alpha})$ if there exist
constants $c>0$ and $T_0 \ge 0$ such that
$f(T)\le c\,T^{-\alpha},$ 
$\forall\, T\ge T_0.$
A function $f(T)$ is said to be $f(T)=\Omega(T^{-\alpha})$ if there 
exist constants $c>0$ and $T_0\ge 0$ such that
$f(T)\ge c\,T^{-\alpha},\
\forall\, T\ge T_0.$

\section{Problem formulation}
\label{sec:prob_formulation}

We work under the following standard regularity assumption for the loss function 
in~\eqref{eq:optimization_main}, which has been widely adopted in the related 
literature (see, e.g.,~\cite{EH:16,YZ-YZ-KJ-MZ:22}).

\begin{assumption}[Properties of the loss]
\label{as:smoothness}
For each $t \in \naturalnneg$, the gradient of $\Phi_t(\cdot)$ is Lipschitz 
continuous with constant $L_{\nabla \Phi}$. 
Moreover, $\Phi_t(\cdot)$ is bounded from below\footnote{When each function 
$\Phi_t(\cdot)$ is bounded below by $\Phi_{\mathrm{low},t}$ we take 
$\Phi_{\mathrm{low}}$ as the minimum value.} by
$\Phi_{\mathrm{low}} \in \mathbb{R}$.
\QEDB \end{assumption}


For $t,T\in\naturalnneg$ and $u_t \in \real^p$, define the one-step temporal drift of the objective in~\eqref{eq:optimization_main} and its cumulative counterpart as
\begin{align}\label{eq:objective_drift}
d_t
&\defeq
\sup_{u \in \mathbb{R}^p} |\Phi_{t+1}(u)-\Phi_t(u)|,
&
D_T
&\defeq
\sum_{t=0}^{T-1} d_t .
\end{align}
Throughout, we assume\footnote{In many control applications, the input $u$ is constrained to a bounded set $\mathcal U\subset\real^p$, in which case one may instead define $d_t\defeq\sup_{u\in\mathcal U}|\Phi_{t+1}(u)-\Phi_t(u)|$. Here, we consider the unconstrained domain $\real^p$ to avoid introducing the projection step required to enforce $u_t\in\mathcal U$.} that $d_t < \infty$ for all $t \in \naturalnneg$, which implies $D_T < \infty$ for any finite $T \in \naturalnneg$.

\begin{example}[Illustrative drift models]
We provide representative instances of $\Phi_t$ and the associated drift
$d_t$ in~\eqref{eq:objective_drift}.

\emph{(i) Time-varying exogenous signal.}
Consider $\Phi_t(u)=\Psi(u,w_t),$ where $w_t \in \real^m$ is a time-varying 
disturbance or reference signal.
If $\Psi$ is $L_{\Psi,w}$-Lipschitz in its second argument, then
\[
d_t
=
\sup_{u \in \mathbb{R}^p}
|\Psi(u,w_{t+1})-\Psi(u,w_t)|
\le
L_{\Psi,w}\|w_{t+1}-w_t\|.
\]
This setting naturally appears in feedback optimization problems (see 
\Cref{sec:feedback_opt_description} for a more detailed discussion). 

\emph{(ii) Time-varying linear perturbation.}
Consider
$\Phi_t(u)=\Phi(u)+a_t^\top u,$
where $a_t \in \real^p$ is time varying. Then
\[
d_t
=
\sup_{u \in \real^p}
|(a_{t+1}-a_t)^\top u|.
\]
Thus
$d_t
\le
\sup_{u \in \real^p}\|u\|
\,\|a_{t+1}-a_t\|.$

\emph{(iii) Moving minimizer model.}
Consider
$\Phi_t(u)=\varphi(u-u_t^\star),$
where $\varphi:\real^p \to \real$ is $L_\varphi$-Lipschitz continuous and
$u_t^\star$ denotes a time-varying optimizer. Then
\[
d_t
=
\sup_{u \in \mathcal{U}}
|\varphi(u-u_{t+1}^\star)-\varphi(u-u_t^\star)|
\le
L_\varphi
\|u_{t+1}^\star-u_t^\star\|.
\]
This model is standard in time-varying optimization, where the optimal 
solution trajectory drifts over time.
\QEDB
\end{example}







In this work, we focus on the \emph{bandit} setting; that is, scenarios where
algorithms to solve~\eqref{eq:optimization_main} may have access only to 
zeroth-order evaluations of the objective:
\begin{align}\label{eq:oracle_tv_problem}
(t, u) \mapsto \Phi_t(u).
\end{align}
Nonconvex optimization under bandit feedback has been extensively studied in 
the literature; see, e.g., \cite{AA-OD-LX:10,YZ-YZ-KJ-MZ:22}. In many 
applications, however, exact oracle evaluations of the 
form~\eqref{eq:oracle_tv_problem} are unavailable
(see Example~\ref{ex:illustrative_inexact_oracles} and 
\Cref{sec:feedback_opt_description} for representative examples). Instead, one 
only has access to an \emph{inexact} zeroth-order oracle as in 
\eqref{eq:oracle_inexact}. We assume that $\tilde{\Phi}_t(\cdot)$ is continuous 
for each $t \in \naturalnneg$.

In analogy with~\eqref{eq:objective_drift}, for 
$t,T\in\naturalnneg$, and $u \in \real^p$, we define the oracle error sequence and its cumulative counterpart as
\begin{equation}
    b_t
    :=
    \sup_{u\in\mathbb{R}^p}
    \left|
        \widetilde{\Phi}_t(u)-\Phi_t(u)
    \right|,
    \qquad
    B_T
    :=
    \sum_{t=0}^{T-1}b_t.
    \label{as:inexact_oracle}
\end{equation}
Throughout, we assume that \(b_t<\infty\) for all \(t\in\naturalnneg\); 
consequently, \(B_T<\infty\) for every finite \(T\in\naturalnneg\).

\begin{example}[Illustrative oracle-error models]
\label{ex:illustrative_inexact_oracles}
We provide representative instances of the oracle approximation 
$\tilde{\Phi}_t$.

\emph{(i) Gradient-free smoothing oracle.}
Suppose the oracle is constructed from a smoothed version of the objective, as commonly done in bandit optimization \cite{AF-AK-HM:04,XC-KJRL:19}:
$\tilde{\Phi}_t(u)
=
\mathbb E_v\!\left[\Phi_t(u+\delta v)\right],$
where $v$ is a random perturbation and $\delta>0$ is a smoothing radius. 
If $\Phi_t$ is $L$-Lipschitz continuous, then
\[
b_t
= 
\sup_{u\in\mathbb{R}^p}\left|
\mathbb E_v[\Phi_t(u+\delta v)]-\Phi_t(u)
\right|
\le
L\,\delta\,\mathbb E[\|v\|].
\]

\emph{(ii) Feedback optimization with transient measurements.}
Consider a control setting in which $\Phi_t(u_t)$ measures the performance of a 
plant in steady state; that is 
$\Phi_t(u_t)=\Psi(u_t,y_{\rm ss}(u_t,w_t)),$
where $y_{\mathrm{ss}}(u_t,w_t)$ models the input-to-output response of the 
plant at steady state and $w_t$ represents unknown disturbances. Because $w_t$
is unknown, one could approximate the steady-state plant output with the 
instantaneous output; that is, $y_{\mathrm{ss}}(u_t,w_t)\approx y_{t+1}$, leading to the oracle approximation
$\widetilde\Phi_t(u_t)=\Psi(u_t,y_{t+1}).$ 
This pointwise error is studied in detail in \Cref{sec:feedback_opt_description}.

\emph{(iii) Adversarial oracle corruption.}
Suppose the oracle is corrupted by an unknown time-varying perturbation:
\[
\tilde{\Phi}_t(u)
=
\Phi_t(u)+\Delta_t(u),
\]
where $\Delta_t(u)$ is selected adversarially, possibly to deteriorate the
performance of the optimization algorithm. Supposing that the adversary has 
limited perturbation capabilities, $|\Delta_t(u)| \leq \beta_t,$
and therefore $b_t
=
\sup_{u\in\mathbb{R}^p}|\Delta_t(u)|
\leq \beta_t$.
This setting is common in online convex optimization and robust learning,
where the objective sequence may be generated by a strategic environment.
\QEDB
\end{example}

We formalize the problem of interest as follows.

\begin{problem}[Design direct-search methods for time-varying optimization under bandit feedback]
\label{prob:time_varying_inexact_design}
Design an optimization algorithm, relying on inexact zeroth-order oracle 
information \eqref{eq:oracle_inexact}, to solve the time-varying optimization 
problem~\eqref{eq:optimization_main}.
\QEDB
\end{problem}

\begin{algorithm}[t]
\caption{Online two-point direct search}
\label{alg:two_point_ds_tv}
\begin{algorithmic}[1]
\Require Initial iterate $u_0 \in \real^p$, probing ratios 
$\{\delta_t\}_{t \ge 0}$, oracle 
$\tilde \Phi_t$

\State $t \gets 0$
\While{stopping criterion not met}
    \State \label{alg:two_point_ds_tv:l1} Evaluate $\tilde \Phi_t(u_t)$

    \State \label{alg:two_point_ds_tv:l2} Draw 
    $v_t \sim \mathcal{N}(0,\frac{1}{p}I_p)$ and define $u_{t+1} := u_t + \delta_t v_t$
    
    \State 
    \label{alg:two_point_ds_tv:l3}
    Evaluate 
    $\tilde \Phi_{t+1}(u_{t+1})$

    \State \label{alg:two_point_ds_tv:l4} Update the decision as
    $    u_{t+2} =
        \begin{cases}
            u_{t+1}, & \text{if } \tilde \Phi_{t+1}(u_{t+1}) \le \tilde \Phi_t(u_t), \\
            u_t, & \text{otherwise.}
        \end{cases}$

    \State \label{alg:two_point_ds_tv:l5} $t \gets t + 2$
\EndWhile
\end{algorithmic}
\end{algorithm}


\section{Online direct search method for time-varying optimization}
\label{sec:tv_algorithms}

\subsection{Method description and stationarity guarantees}
A direct search method that addresses \Cref{prob:time_varying_inexact_design} 
is presented in \Cref{alg:two_point_ds_tv}.  
At each iteration, the algorithm first evaluates the objective function at the 
\textit{current decision} $u_t$ (line~\ref{alg:two_point_ds_tv:l1}).  
It then generates a random search direction $v_t$, constructs a 
\textit{candidate decision} by perturbing $u_t$ along this direction with 
\textit{probing radius} (or \textit{exploration ratio}) $\delta_t$, and 
evaluates the objective at this perturbed point 
(line~\ref{alg:two_point_ds_tv:l2}).  
An accept/reject step follows, where the candidate decision is accepted when it 
yields a lower objective value than the current decision 
(line~\ref{alg:two_point_ds_tv:l4}) or it is rejected otherwise. The procedure 
then advances the time index by two steps (line~\ref{alg:two_point_ds_tv:l5}), 
reflecting the two function evaluations performed per iteration.

The following result establishes a convergence guarantee for 
\Cref{alg:two_point_ds_tv}.

\begin{theorem}[Stationarity guarantees of \Cref{alg:two_point_ds_tv} under 
temporal drift and inexact oracle]
\label{thm:two_point_tv_inexact}
Suppose that Assumption~\ref{as:smoothness} holds. For any $T \in \Neventwo$, 
the iterates generated by \Cref{alg:two_point_ds_tv} satisfy:
\begin{align}
\label{eq:two_point_main_bound}
\frac{1}{\sqrt{2\pi p}}
\sum_{t \in \evenset{T}} &
\delta_t\,\mathbb{E}\!\left[\|\nabla \Phi_t(u_t)\|\right]
\le
\Phi_0(u_0)-\Phi_{\mathrm{low}} 
+\frac{L_{\nabla \Phi}}{4}
\sum_{t \in \evenset{T}} \delta_t^2
+2 D_T
+2 B_T.
\end{align}
Consequently,
\begin{align}
\label{eq:two_point_main_cor}
\min_{t \in \evenset{T}}
\mathbb{E}\!\left[\|\nabla \Phi_t(u_t)\|\right]
&\le
\frac{\sqrt{2\pi p}}
{\sum_{t \in \evenset{T}} \delta_t}
\Bigg[
\Phi_0(u_0)-\Phi_{\mathrm{low}}
+\frac{L_{\nabla \Phi}}{4}
\sum_{t \in \evenset{T}} \delta_t^2
+2 D_T
+2 B_T
\Bigg].
\end{align}
\end{theorem}

\begin{proof}
Fix $t \in \Nevennneg$ and define
\begin{align*}
\tilde{\Phi}_{\min,t}
\defeq
\min\{\tilde{\Phi}_t(u_t),\,\tilde{\Phi}_{t+1}(u_{t+1})\}, &&
u_{t+1} \defeq u_t + \delta_t v_t.
\end{align*}
We first relate $u_{t+2}$ to $\tilde{\Phi}_{\min,t}$. If the step is accepted, 
then $u_{t+2}=u_{t+1}$ and, by \Cref{eq:objective_drift,as:inexact_oracle},
\begin{align*}
\Phi_{t+2}(u_{t+2})
=
\Phi_{t+2}(u_{t+1})
\le
\Phi_{t+1}(u_{t+1}) + d_{t+1}
&\le
\tilde{\Phi}_{t+1}(u_{t+1}) + b_{t+1} + d_{t+1}
\notag\\
&=
\tilde{\Phi}_{\min,t} + b_{t+1} + d_{t+1}.
\end{align*}
If the step is rejected, then $u_{t+2}=u_t$ and, similarly,
\begin{align*}
\Phi_{t+2}(u_{t+2})
=
\Phi_{t+2}(u_t)
\le
\Phi_t(u_t) + d_t + d_{t+1}
&\le
\tilde{\Phi}_t(u_t) + b_t + d_t + d_{t+1}
\notag\\
&=
\tilde{\Phi}_{\min,t} + b_t + d_t + d_{t+1}.
\end{align*}
Hence, in both cases,
\begin{equation}
\Phi_{t+2}(u_{t+2})
\le
\tilde{\Phi}_{\min,t}
+ b_t + b_{t+1}
+ d_t + d_{t+1}.
\label{eq:selected-vs-min}
\end{equation}
Therefore, by \Cref{as:inexact_oracle},
\begin{align}
\tilde{\Phi}_{\min,t}
&\le
\min\{\Phi_t(u_t),\,\Phi_{t+1}(u_{t+1})\}
+ \max\{b_t,b_{t+1}\}
\notag\\
&\le
\min\{\Phi_t(u_t),\,\Phi_{t+1}(u_{t+1})\}
+ b_t + b_{t+1}.
\label{eq:min-inexact}
\end{align}
Combining \eqref{eq:selected-vs-min} and \eqref{eq:min-inexact} yields
\begin{equation}
\Phi_{t+2}(u_{t+2})
\le
\min\{\Phi_t(u_t),\,\Phi_{t+1}(u_{t+1})\}
 + 2(b_t + b_{t+1})
+ d_t + d_{t+1}.
\label{eq:key-min}
\end{equation}

Now, using \Cref{eq:objective_drift},  we have
\[
\Phi_{t+1}(u_{t+1}) \le \Phi_t(u_{t+1}) + d_t,
\]
and thus~\eqref{eq:key-min} can be rewritten as
\begin{equation}
\Phi_{t+2}(u_{t+2})
\le
\min\{\Phi_t(u_t),\,\Phi_t(u_{t+1})\}
+ 2(b_t + b_{t+1})
+ 2d_t + d_{t+1}.
\label{eq:key-static}
\end{equation}

Next, let $g_t \defeq \nabla \Phi_t(u_t)$. By smoothness,
\[
\Phi_t(u_t)-\Phi_t(u_{t+1})
\ge
-\delta_t \langle g_t,v_t\rangle
-\frac{L_{\nabla \Phi}}{2}\delta_t^2\|v_t\|^2 .
\]
Therefore,
\begin{align}
\bigl(\Phi_t(u_t) -\Phi_t(u_{t+1})\bigr)_+ 
& \ge
\left(
-\delta_t \langle g_t,v_t\rangle
-\frac{L_{\nabla \Phi}}{2}\delta_t^2\|v_t\|^2
\right)_+ \notag\\
&\ge
\left(
-\delta_t \langle g_t,v_t\rangle
-\frac{L_{\nabla \Phi}}{2}\delta_t^2\|v_t\|^2
\right)
\mathbf 1_{\{\langle g_t,v_t\rangle<0\}} .
\end{align}
Let $\mathcal F_{t-2} \defeq \sigma(v_0,v_2,\ldots,v_{t-2})$ denote the 
sigma-algebra generated by all search directions prior to time \(t\), and note 
that, conditionally on $\mathcal F_{t-2}$, \(u_t\) and \(g_t\) are
\(\mathcal F_{t-2}\)-measurable, while \(v_t\) is independent of
\(\mathcal F_{t-2}\) and satisfies \(v_t\sim\mathcal N(0,\frac1p I_p)\).
Taking conditional expectation and applying 
\Cref{lem:gaussian_one_sided} gives
\begin{align}
\mathbb{E}\!\left[
\bigl(\Phi_t(u_t)-\Phi_t(u_{t+1})\bigr)_+
\,\middle|\, \mathcal F_{t-2}
\right]
&\ge
\frac{\delta_t}{\sqrt{2\pi p}}\|g_t\|
-\frac{L_{\nabla \Phi}}{4}\delta_t^2 .
\end{align}
Substituting into \eqref{eq:key-static}, taking conditional expectation with
respect to \(\mathcal F_{t-2}\), and then taking expectation over \(v_t\),
yields
\begin{align}\label{eq:main-step}
\mathbb{E}[\Phi_{t+2}(u_{t+2})]
&\le
\mathbb{E}[\Phi_t(u_t)]
-\frac{\delta_t}{\sqrt{2\pi p}}
\mathbb{E}\!\left[\|\nabla \Phi_t(u_t)\|\right]\\
&\quad
+\frac{L_{\nabla \Phi}}{4}\delta_t^2
+2d_t + d_{t+1}
+2 b_t + 2 b_{t+1}.\notag
\end{align}
Summing \eqref{eq:main-step} over $t \in \evenset{T}$ yields a telescoping sum. 
Moreover,
\begin{align*}
\sum_{t \in \evenset{T}} (2d_t + d_{t+1})
\le
2 \sum_{t=0}^{T-1} d_t
= 2 D_T,
\end{align*}
and similarly
\[
\sum_{t \in \evenset{T}} 2 (b_t + b_{t+1})
\le
2 \sum_{t=0}^{T-1}  b_t 
= 2B_T.
\]
Using $\Phi_T(u_T) \ge \Phi_{\mathrm{low}}$ and taking expectation
yields \eqref{eq:two_point_main_bound}. 
Finally, the bound~\eqref{eq:two_point_main_cor} follows by standard
averaging arguments: since all terms in the sum are nonnegative,
we have
\[
\min_{t \in \evenset{T}} 
\mathbb{E}\!\left[\|\nabla \Phi_t(u_t)\|\right]
\le
\frac{
\sum_{t \in \evenset{T}} 
\delta_t\,\mathbb{E}\!\left[\|\nabla \Phi_t(u_t)\|\right]
}{
\sum_{t \in \evenset{T}} \delta_t
}.
\]
Combining this inequality with~\eqref{eq:two_point_main_bound} yields
\eqref{eq:two_point_main_cor}.
\end{proof}

\Cref{thm:two_point_tv_inexact} shows that \Cref{alg:two_point_ds_tv} achieves 
a descent in expectation, quantified through a bound on the cumulative sum of 
the gradient norms evaluated along the iterates. Importantly, the theorem 
provides an estimate on the minimum expected gradient norm along the 
trajectory, showing that the algorithm tracks approximate stationary points up 
to errors induced by time variations and oracle inaccuracies.



Theorem~\ref{thm:two_point_tv_inexact} admits several important
specializations:
\begin{itemize}
        \item \textit{No temporal drift.}
    If $\Phi_{t+1}(u)=\Phi_t(u)$  $\forall t$, then $d_t\equiv0$ and
    \eqref{eq:two_point_main_bound} reduces to
    \begin{align*}
    \frac{1}{\sqrt{2\pi p}}
    \sum_{t \in \evenset{T}}
    \delta_t\,\mathbb{E}\!\left[\|\nabla \Phi(u_t)\|\right]
    & \le
    \Phi_0(u_0)-\Phi_{\mathrm{low}} 
    +\frac{L_{\nabla \Phi}}{4}
    \sum_{t \in \evenset{T}} \delta_t^2
    +2B_T.
    \end{align*}

    \item \textit{Exact oracle.}
    If $\tilde \Phi_t(u)=\Phi_t(u)$  $\forall t$, then $b_t\equiv0$ and
    \eqref{eq:two_point_main_bound} reduces to
    \begin{align*}
    \frac{1}{\sqrt{2\pi p}}
    \sum_{t \in \evenset{T}}
    \delta_t\,\mathbb{E}\!\left[\|\nabla \Phi_t(u_t)\|\right]
    & \le
    \Phi_0(u_0)-\Phi_{\mathrm{low}} 
    +\frac{L_{\nabla \Phi}}{4}
    \sum_{t \in \evenset{T}} \delta_t^2
    +2 D_T.
    \end{align*}

    \item \textit{No temporal drift and exact oracle.}
    If both $b_t\equiv0$ and $d_t\equiv0$, then
    \eqref{eq:two_point_main_bound} simplifies to
    \begin{align*}
    \frac{1}{\sqrt{2\pi p}}
    \sum_{t \in \evenset{T}}
    \delta_t\,\mathbb{E}\!\left[\|\nabla \Phi(u_t)\|\right]
    & \le
    \Phi_0(u_0)-\Phi_{\mathrm{low}}
    +\frac{L_{\nabla \Phi}}{4}
    \sum_{t \in \evenset{T}} \delta_t^2,
    \end{align*}
    recovering the form of standard stationarity guarantee for direct search
    methods in static settings \cite{EB-EG-PR:20}.
    Note that, relative to \cite{EB-EG-PR:20}, our bound deteriorates by a factor 
    of $\sqrt{2}$ due to the use of a two-point accept-reject rule (in place of a 
    three point comparison as in \cite{EB-EG-PR:20}).
\end{itemize}

\subsection{Constant probing}

We next identify a stepsize for which \Cref{alg:two_point_ds_tv} reaches an 
\(\varepsilon\)-stationary point and establish its iteration complexity.

\begin{theorem}[Iteration complexity of \Cref{alg:two_point_ds_tv}--constant probing ratio]
\label{thm:two_point_const}
Suppose that Assumption~\ref{as:smoothness} holds. 
Fix $\varepsilon > 0$ and let the probing ratio be
\begin{align}\label{eq:constant_delta}
\delta_t \equiv \delta := \frac{4\varepsilon}{3\sqrt{2\pi p}\,L_{\nabla \Phi}}.
\end{align}
Let $T \in \Neventwo$, 
$\mathcal{E}_T \defeq \frac{2 D_T +2B_T}{T}$, 
and suppose that
\begin{align}\label{eq:mu-condition-thm2-tv}
\mathcal{E}_T
&\le
\frac{\varepsilon^2}{9\pi p L_{\nabla \Phi}}.
\end{align}
Then, for any $T$ satisfying
\begin{align}
T\ge
9\pi p L_{\nabla \Phi}\bigl(\Phi_0(u_0)-\Phi_{\mathrm{low}}\bigr)\varepsilon^{-2},
\label{eq:T-condition-thm2-tv}
\end{align}
the iterates of \Cref{alg:two_point_ds_tv} satisfy
\begin{equation*}
\min_{t \in \evenset{T}}
\mathbb{E}\!\left[
\|\nabla \Phi_t(u_t)\|
\right]
\le \varepsilon.
\end{equation*}
\end{theorem}

\begin{proof}
Under the stated assumptions, we have
\[
\sum_{t \in \evenset{T}}
(2d_t+d_{t+1})
\le
2 D_T,
\qquad
2B_T = 2\sum_{t=0}^{T-1} b_t .
\]
Using \Cref{eq:two_point_main_cor} with $\delta_t \equiv \delta$, we obtain
\begin{align}\label{eq:constant_delta_bound}
\min_{t=0,2,\dots,T-2} &
\mathbb{E}\!\left[\|\nabla \Phi_t(u_t)\|\right]
\le
\frac{\sqrt{2\pi p}}{\frac{T}{2}\delta}
\Bigg[
\Phi_0(u_0)-\Phi_{\mathrm{low}}
+\frac{L_{\nabla \Phi}}{4}\frac{T}{2}\delta^2
+2 D_T
+2B_T
\Bigg].
\end{align}
Simplifying yields
\begin{align}\label{eq:main-cor-const-tv}
\min_{t=0,2,\dots,T-2} &
\mathbb{E}\!\left[\|\nabla \Phi_t(u_t)\|\right]
\le
\sqrt{2\pi p}
\Bigg[
\frac{2(\Phi_0(u_0)-\Phi_{\mathrm{low}})}{\delta T}
+\frac{L_{\nabla \Phi}}{4}\delta
+\frac{4D_T+4B_T}{\delta T}
\Bigg].
\end{align}
We next bound each term individually. First, \eqref{eq:constant_delta} implies
\[
\sqrt{2\pi p}\,\frac{L_{\nabla \Phi}}{4}\delta
=
\frac{\varepsilon}{3}.
\]
Next, using $\mathcal{E}_T = \frac{2 D_T +2B_T}{T}$, we have
\[
\sqrt{2\pi p}\,\frac{2(2 D_T +2B_T)}{\delta T}
=
\sqrt{2\pi p}\,\frac{2\mathcal{E}_T}{\delta}
=
\frac{3\pi p L_{\nabla \Phi}\,\mathcal{E}_T}{\varepsilon}.
\]
Thus, condition \eqref{eq:mu-condition-thm2-tv} implies
\[
\sqrt{2\pi p}\,\frac{2(2 D_T +2B_T)}{\delta T}
\le
\frac{\varepsilon}{3}.
\]
Finally,
\[
\sqrt{2\pi p}\,\frac{2(\Phi_0(u_0)-\Phi_{\mathrm{low}})}{T\delta}
=
\frac{3\pi p L_{\nabla \Phi}\bigl(\Phi_0(u_0)-\Phi_{\mathrm{low}}\bigr)}{T\varepsilon}.
\]
Hence, condition \eqref{eq:T-condition-thm2-tv} ensures
\[
\sqrt{2\pi p}\,\frac{2(\Phi_0(u_0)-\Phi_{\mathrm{low}})}{T\delta}
\le
\frac{\varepsilon}{3}.
\]
Substituting into \eqref{eq:main-cor-const-tv} yields the claim.
\end{proof}

By \Cref{thm:two_point_const},  \Cref{alg:two_point_ds_tv}
attains an $\varepsilon$-approximate stationary point within
\begin{align}\label{eq:iteration_complex_constant_delta}
T
=
O\!\left(
p\varepsilon^{-2}
\right),
\end{align}
iterations, provided that the probing ratio is chosen according to
\eqref{eq:constant_delta} and that the combined effect of temporal variability 
and oracle inexactness (captured by $\mathcal{E}_T$) is sufficiently small. 
This generalizes the classical complexity scaling of static direct-search methods 
in both the dimension $p$ and the target accuracy 
$\varepsilon$~\cite{EB-EG-PR:20}.


For a fixed horizon $T$, the iteration complexity estimate
\eqref{eq:iteration_complex_constant_delta} can be used to characterize the
smallest accuracy level certified by the theorem. This is given by:
\begin{align}\label{eq:resolution_constant}
\varepsilon
\ge
\sqrt{\frac{9\pi p L_{\nabla \Phi}
\bigl(\Phi_0(u_0)-\Phi_{\mathrm{low}}\bigr)}{T}},
\end{align}
showing that the achievable resolution after $T$ iterations scales 
as $\varepsilon = \Omega(T^{-1/2})$.


Condition~\eqref{eq:mu-condition-thm2-tv} imposes a requirement, asking the 
average problem variability and oracle inexactness to remain sufficiently 
small. Precisely, it requires:
\begin{align}
\label{eq:target_ET_constant}
\mathcal{E}_T
=
O\!\left(
\frac{\varepsilon^2}{p}
\right),
\end{align}
which asks the admissible cumulative temporal drift and oracle inexactness to  
decrease quadratically with the target stationarity accuracy $\varepsilon$, and 
inversely proportionally to the problem dimension $p$.
More precisely, recalling that
$\mathcal{E}_T
=
\frac{1}{T} (2  D_T+2 B_T),$
condition~\eqref{eq:mu-condition-thm2-tv} imposes the requirement
\begin{align}\label{eq:bound_cumulat_drift_constant_delta}
2  D_T+2 B_T
\le
\frac{\varepsilon^2 T}{9 \pi p L_{\nabla \Phi}},
\end{align}
imposing the cumulative drift and oracle inexactness grow at most 
linearly with the horizon $T$ (equivalently, their average contribution must 
scale as $O\!\left(\varepsilon^2 / p\right)$).

\begin{example}[Uniform bounds on drift and inexactness]
\label{ex:uniform_bound_constant_delta}
In the special case where $d_t \le d$ and $b_t \le b$ for all $t$, we have 
$\mathcal{E}_T \le 2 d + 2b$, and the requirement~\eqref{eq:mu-condition-thm2-tv} 
simplifies to:
\begin{align}\label{eq:d_plus_b_bound_constant_delta}
d + b
\le
\frac{\varepsilon^2}{18 \pi p L_{\nabla \Phi}}.
\end{align}
Equivalently, this can be interpreted as a lower bound on the attainable 
resolution, yielding
$\varepsilon
\ge
3\sqrt{2\pi p L_{\nabla \Phi}(d + b)}.$
\QEDB\end{example}

\subsection{Diminishing probing}

The use of a constant probing ratio, as in
\Cref{thm:two_point_const}, comes with a key limitation: the choice of
\(\delta\) depends on the desired accuracy \(\varepsilon\). Consequently, the target accuracy \(\varepsilon\) must be specified before executing the
algorithm. The use of 
a diminishing probing ratio overcomes this limitation.

\begin{theorem}[Iteration complexity of \Cref{alg:two_point_ds_tv}--diminishing probing ratio]
\label{thm:two_point_diminishing}
Suppose that Assumption~\ref{as:smoothness} holds. 
Fix $\varepsilon > 0$ and let the probing ratio be
\begin{align}\label{eq:vanishing_delta}
\delta_t = 1/\sqrt{t+1}.
\end{align}
Let $T \in \Neventwo$, 
$\mathcal{E}_T \defeq \frac{2 D_T +2B_T}{T}$, 
and suppose that
\begin{equation}\label{eq:E_condition_diminishing}
\mathcal{E}_T
\;\le\;
\frac{\varepsilon}{2\sqrt{2\pi p}(\sqrt{2}+1)\sqrt{T}}.
\end{equation}
Then, for any $T$ satisfying
\begin{align}\label{eq:T_condition_diminishing}
T \;\ge\;
& \max\!\Bigg\{
\frac{2\pi p(\sqrt{2}+1)^2}{\varepsilon^2}
\bigl(4(\Phi_0(u_0)-\Phi_{\mathrm{low}})+L_{\nabla \Phi}\bigr)^2,
\notag\\
&\qquad \frac{8\pi p(\sqrt{2}+1)^2 L_{\nabla \Phi}^2}{\varepsilon^2}
\log^2\!\Bigl(
\frac{2(\sqrt{2}+1)\sqrt{2\pi p}\,L_{\nabla \Phi}}{\varepsilon}
\Bigr)
\Bigg\},
\end{align}
the iterates generated by \Cref{alg:two_point_ds_tv} satisfy
\begin{equation*}
\min_{t \in \evenset{T}}
\mathbb{E}\!\left[
\|\nabla \Phi_t(u_t)\|
\right]
\le \varepsilon.
\end{equation*}
\end{theorem}

\begin{proof}
Using \Cref{eq:two_point_main_cor}, we obtain
\begin{align}
\min_{t=0,2,\dots,T-2}
\mathbb{E}\!\left[\|\nabla \Phi_t(u_t)\|\right]
&\le
\frac{\sqrt{2\pi p}}
{\sum_{t \in \evenset{T}}\delta_t}
\Bigg[
\Phi_0(u_0)-\Phi_{\mathrm{low}}
\notag\\
&\quad
+\frac{L_{\nabla \Phi}}{4}
\sum_{t \in \evenset{T}}\delta_t^2
+2D_T
+2B_T
\Bigg].
\label{eq:main-cor-diminishing}
\end{align}
By \Cref{lem:diminishing_stepsize_sums}, the chosen probing ratios satisfy
\[
\sum_{t \in \evenset{T}}\delta_t
\ge \frac{\sqrt{T}}{\sqrt{2}+1},
\qquad
\sum_{t \in \evenset{T}}\delta_t^2
\le 1+\frac{1}{2}\log(T).
\]
Substituting into \eqref{eq:main-cor-diminishing} yields
\begin{align}\label{eq:split_terms}
\min_{t \in \evenset{T}}
\mathbb{E}\!\left[\|\nabla \Phi_t(u_t)\|\right]
&\le
\sqrt{2\pi p}(\sqrt{2}+1)
\Bigg[
\frac{\Phi_0(u_0)-\Phi_{\mathrm{low}}}{\sqrt{T}}
\notag\\
&\quad
+
\frac{\frac{L_{\nabla \Phi}}{4}\!\left(1+\frac{1}{2}\log T\right)}{\sqrt{T}}
+
\frac{2 D_T +2B_T}{\sqrt{T}}
\Bigg].
\end{align}

We now bound the three terms separately. Define
\[
\alpha = \frac{\varepsilon}{(\sqrt{2}+1)\sqrt{2\pi p}}.
\]
By iterating the arguments in the proof of \Cref{thm:two_point_const}, we 
conclude that condition \eqref{eq:T_condition_diminishing} ensures that the 
first two terms in \eqref{eq:split_terms} are upper bounded by $\varepsilon/2$. 
Moreover, using the definition $\mathcal{E}_T = \frac{2 D_T +2B_T}{T},$ we can 
rewrite the third term as
\[
\frac{2 D_T +2B_T}{\sqrt{T}}
=
\mathcal{E}_T \sqrt{T}.
\]
Hence, condition \eqref{eq:E_condition_diminishing} ensures that the third term is also upper bounded by $\varepsilon/2$. Combining the bounds yields the claim.
\end{proof}

\Cref{thm:two_point_diminishing} shows that, as\footnote{as $\varepsilon \to 0$, the second term in \eqref{eq:T_condition_diminishing} dominates the maximization.} $\varepsilon \to 0$, \Cref{alg:two_point_ds_tv} 
attains an $\varepsilon$-approximate stationary point within
\begin{align}\label{eq:iteration_complexity_diminishing}
T
=
O\!\left(
\frac{p}{\varepsilon^2}
\log^2\!\Bigl(\frac{L_{\nabla \Phi}\sqrt{p}}{\varepsilon}\Bigr)
\right),
\end{align}
iterations. 
This dependence is illustrated in \Cref{fig:dim_p}. 
Notice that the use of diminishing probing yields an additional logarithmic 
factor relative to \eqref{eq:iteration_complex_constant_delta}, corresponding to a 
mild performance degradation.


\begin{remark}[Accuracy level as a function of $T$]
\label{rem:resolution_diminishing}
For a fixed horizon $T$, the iteration complexity estimate
\eqref{eq:iteration_complexity_diminishing} can be used to characterize the
smallest accuracy level certified by the theorem. This is given by:
\begin{align}\label{eq:resolution_diminishing}
\varepsilon
=
\Omega\!\left(
\frac{\sqrt{p}\log T}{\sqrt{T}}
\right).
\end{align}
To see this, rewrite \eqref{eq:iteration_complexity_diminishing} as
$\varepsilon
\ge
\sqrt{\frac{Cp}{T}}
\log\!\Bigl(
\frac{L_{\nabla \Phi}\sqrt{p}}{\varepsilon}
\Bigr),$
for some $C>0$. Since this relation is implicit in
$\varepsilon$, we verify \eqref{eq:resolution_diminishing} by substitution.
Setting
$\varepsilon
=
c\sqrt{\frac{p}{T}}\log T,$
for some $c>0$, yields
$\frac{p}{\varepsilon^2}
=
\frac{T}{c^2\log^2 T},$
and, substituting into the right-hand side above, gives
$\log\!\Bigl(
\frac{L_{\nabla \Phi}\sqrt{p}}{\varepsilon}
\Bigr)
=
\frac{1}{2}\log T-\log\log T+O(1).$
Therefore,
\[
\frac{p}{\varepsilon^2}
\log^2\!\Bigl(
\frac{L_{\nabla \Phi}\sqrt{p}}{\varepsilon}
\Bigr)
=
O(T),
\]
showing that \eqref{eq:resolution_diminishing} satisfies
\eqref{eq:iteration_complexity_diminishing} for all sufficiently large $T$ and
a sufficiently large constant $c$. Hence, the smallest accuracy level
certified by the theorem scales as
$\Omega(\sqrt{p/T}\log T)$.
\QEDB\end{remark}

Similarly to \Cref{thm:two_point_const}, the claim requires the combined effect
of temporal drift and oracle inexactness (quantified by \(\mathcal E_T\)) to
remain sufficiently small. Substituting the scaling for $T$ in 
\eqref{eq:iteration_complexity_diminishing} into 
\eqref{eq:E_condition_diminishing} yields the following explicit requirement:
\[
\mathcal{E}_T
=
O\!\left(
\frac{\varepsilon^2}
{
p\,
\log\!\Bigl(
\frac{L_{\nabla \Phi}\sqrt{p}}{\varepsilon}
\Bigr)
}
\right),
\]
which asks that the admissible level of cumulative temporal drift and oracle
inexactness must decrease faster than quadratically with the target 
stationarity accuracy $\varepsilon$. Note that this is a more restrictive requirement relative 
to \eqref{eq:target_ET_constant}, obtained for constant probing ratios. 
Another key difference with respect to the constant probing-ratio regime, where 
the cumulative drift and oracle error are allowed to grow linearly with $T$
(see~\eqref{eq:bound_cumulat_drift_constant_delta}), the diminishing probing 
regime requires (from~\eqref{eq:E_condition_diminishing}):
\begin{align}\label{eq:error_decay_diminishing}
2 D_T +2B_T
\le
\frac{\varepsilon}{2\sqrt{2\pi p}(\sqrt{2}+1)}\,\sqrt{T}.
\end{align}
That is, the cumulative drift and oracle error must grow at most as
\(O(\sqrt{T})\) (equivalently, the average contribution must decay as
\(O(1/\sqrt{T})\)). Compare with \eqref{eq:bound_cumulat_drift_constant_delta}. 
This reflects the increased sensitivity of diminishing probing-ratio schemes to 
temporal variability, which must vanish asymptotically to guarantee convergence.

A direct comparison between the constant and diminishing stepsize cases is 
presented in \Cref{tab:constant_vs_diminishing_rates}. The following example 
specializes these requirements to the case of uniformly bounded drift and 
oracle error.

\setlength{\tabcolsep}{3pt}
\begin{table}[t]
\centering
\caption{Comparison of constant and diminishing probing ratio regimes. Notice that the convergence property holds provided that the requirement on $(D_T+B_T)/T$ is satisfied.}
\label{tab:constant_vs_diminishing_rates}
\begin{tabular}{lccc}
\toprule
\makecell{Probing\\ ratio} & 
\makecell{Requirement on \\ $(D_T+B_T)/T$} &
\makecell{Iteration \\ complexity} &
\makecell{Convergence to\\stationary point}\\
\midrule
Constant &
$O\!\left(\frac{\varepsilon^2}{p}\right)$ &
$O\!\left(\frac{p}{\varepsilon^2}\right)$ & Inexact\\

Diminishing &
$O\!\left(\frac{\varepsilon}{\sqrt{pT}}\right)$ &
$O\!\left(
\frac{p}{\varepsilon^2}
\log^2\!\Bigl(\frac{L_{\nabla \Phi}\sqrt{p}}{\varepsilon}\Bigr)
\right)$ & Exact\\
\bottomrule
\end{tabular}
\end{table}


\begin{example}[Uniform bounds on drift and inexactness]
In the special case where $d_t \le d$ and $b_t \le b$ for all $t$, we have 
$\mathcal{E}_T \le 2d+2b$. In this case, 
\eqref{eq:E_condition_diminishing} reduces to the requirement:
\begin{align}\label{eq:d_plus_b_upper_bound_diminishing}
d+b
\;\le\;
\frac{\varepsilon}
{4\sqrt{2\pi p}(\sqrt{2}+1)\sqrt{T}} .
\end{align}
That is, the instantaneous drift/inexactness must vanish at least as 
$O(1/\sqrt{T})$ to yield asymptotic stationarity.
This can equivalently be interpreted as a lower bound on the attainable 
resolution, yielding
$\varepsilon
\;\ge\;
4\sqrt{2\pi p}(\sqrt{2}+1)(d+b) \sqrt{T}.$
Compare with Example~\ref{ex:uniform_bound_constant_delta}.
\QEDB
\end{example}

The following is a direct consequence of \Cref{thm:two_point_diminishing}.

\begin{corollary}[Asymptotic stationarity with diminishing probing]
\label{cor:two_point_diminishing_asymptotic}
Suppose the assumptions of \Cref{thm:two_point_diminishing} hold. If, further, 
\[
\mathcal{E}_T \sqrt{T} \to 0
\qquad \text{as } T\to+\infty,
\]
then, 
\[
\lim_{T\to+\infty}
\min_{t \in \evenset{T}}
\mathbb{E}\!\left[
\|\nabla \Phi_t(u_t)\|
\right]
=
0.
\]
\end{corollary}

\begin{proof}
By \eqref{eq:split_terms}, there exists a constant $c>0$ such that
\begin{align*}
\min_{t \in \evenset{T}}
\mathbb{E}\!\left[\|\nabla \Phi_t(u_t)\|\right]
&\le
c\Big[
\frac{\Phi_0(u_0)-\Phi_{\mathrm{low}}}{\sqrt{T}}
+
\frac{L_{\nabla \Phi}(1+\log T)}{\sqrt{T}}\\
&\qquad +
\mathcal{E}_T\sqrt{T}
\Big].
\end{align*}
The first two terms vanish as $T\to+\infty$, and the last term vanishes by 
assumption. The claim follows.
\end{proof}

\Cref{cor:two_point_diminishing_asymptotic} shows that, under suitably bounded 
problem variability and oracle inexactness, a diminishing stepsize guarantees 
\textit{exact} asymptotic stationarity. 
\Cref{cor:two_point_diminishing_asymptotic} characterizes an important feature 
of the diminishing stepsize regime, which is the ability to asymptotically 
approach stationary points \textit{exactly}, while constant stepsize schemes 
only guarantee convergence to a neighborhood (namely, \textit{inexact}) of 
stationary points. 


\subsection{Regret analysis}
\label{sec:regret_analysis}

The stationarity guarantees established in the previous sections are stated in
terms of the expected gradient norm. An alternative established performance 
metric is that of \emph{local regret}~\cite{EH-KS-CZ:17}. Specifically, given a 
window length \(w\), $1 \leq w \leq T$, the authors of~\cite{EH-KS-CZ:17}  define the smoothed 
objective $F_{t,w}(u)$
and measure performance through the cumulative quantity $\mathcal R_w(T)$, defined as
\[
F_{t,w}(u)
\defeq
\frac{1}{w}
\sum_{i=0}^{w-1}
\Phi_{t-i}(u), \qquad 
\mathcal R_w(T)
\defeq
\sum_{t=1}^{T}
\bigl\|
\nabla F_{t,w}(u_t)
\bigr\|^2,
\]
which captures stationarity with respect to a temporally averaged objective.
The use of the averaging window is essential in the setting considered 
in~\cite{EH-KS-CZ:17}, as sublinear regret cannot generally be
achieved for the instantaneous objectives. 
In the case \(w=1\), the smoothed objective satisfies \(F_{t,1}=\Phi_t\), and
the local regret of~\cite{EH-KS-CZ:17} reduces to the cumulative squared
gradient norm of the instantaneous objectives, making it directly comparable 
 to the stationarity guarantees developed in this paper. 
Moreover, since \Cref{thm:two_point_tv_inexact} provides guarantees over the 
\textit{even} iterates generated by the two-point scheme, we define the 
corresponding instantaneous local regret as
\[
\mathcal R_1^{\mathrm{even}}(T)
\defeq
\sum_{t\in\evenset{T}}
\mathbb{E}\!\left[
\|\nabla \Phi_t(u_t)\|^2
\right].
\]

It is worth noting another important distinction  of our framework relative to 
standard online optimization settings \cite{EH:16}. In this paper, we model the 
loss' temporal drift explicitly through the cumulative drift measure \(D_T\). 
Although many online optimization frameworks (see, 
e.g.,~\cite{EH-KS-CZ:17,EH:16}) do not impose direct restrictions on the 
objective variation, they assume that each objective \(\Phi_t\) satisfies 
uniform regularity conditions of the form
\[
|\Phi_t(u)| \le M,
\qquad
\|\nabla \Phi_t(u)\| \le G,
\]
for all \(t\).
While these assumptions do not explicitly constrain temporal variation, they
indirectly bound its worst-case growth. Indeed, using these assumptions, we have
\begin{align*}
d_t
\le
\sup_{u \in \mathbb{R}^p} (|\Phi_{t+1}(u)|+|\Phi_t(u)|)
\le
2M.
\end{align*}
Consequently,
$D_T
=
\sum_{t=0}^{T-1} d_t
\le
2MT.$
In other words, frameworks such as~\cite{EH-KS-CZ:17} implicitly assume that 
the cumulative drift grow in the worst-case as $D_T = O(T).$

With these two observations at hand, the following links our framework with 
that of online nonconvex optimization.

\begin{proposition}[Instantaneous local regret]
\label{prop:instantaneous_local_regret}
Suppose that the assumptions of \Cref{thm:two_point_tv_inexact} hold and that
there exists \(G>0\) such that
$\|\nabla \Phi_t(u)\| \le G, \ \forall t,\; \forall u\in\mathbb{R}^p.$
Then, the following bounds hold.
\begin{itemize}
    \item 
\emph{Constant probing.} If \(\delta_t\equiv\delta\), then
\begin{equation*}
\mathcal R_1^{\mathrm{even}}(T)
\le
G\sqrt{2\pi p}
\bigg[
\frac{
\Phi_0(u_0)-\Phi_{\mathrm{low}}+2 D_T+2B_T
}{\delta}
+
\frac{L_{\nabla\Phi}}{8}T\delta
\bigg].
\end{equation*}
In particular, under the choice \eqref{eq:constant_delta},
$\mathcal R_1^{\mathrm{even}}(T)
=
\mathcal{O}\!\left(T\varepsilon +
\frac{p(1+D_T+B_T)}{\varepsilon}
\right).$

\item \emph{Diminishing probing.} If \(\delta_t=\frac1{\sqrt{t+1}}\), then
\begin{equation*}
\mathcal R_1^{\mathrm{even}}(T)
\le
G\sqrt{2\pi pT}
\bigg[
\Phi_0(u_0)-\Phi_{\mathrm{low}}
+
\frac{L_{\nabla\Phi}}{4}(1+\log T)
+2 D_T+2B_T
\bigg].
\end{equation*}
Consequently,
$\mathcal R_1^{\mathrm{even}}(T)
=
\mathcal{O}\!\left(
\sqrt{pT}\,
\bigl(\log T+D_T+B_T\bigr)
\right).$
\end{itemize}
\QEDB\end{proposition}

\begin{proof}
By the bounded-gradient assumption,
\[
\mathbb{E}\!\left[\|\nabla \Phi_t(u_t)\|^2\right]
\le
G\,
\mathbb{E}\!\left[\|\nabla \Phi_t(u_t)\|\right].
\]
Therefore,
\[
\mathcal R_1^{\mathrm{even}}(T)
\le
G
\sum_{t\in\evenset{T}}
\mathbb{E}\!\left[\|\nabla \Phi_t(u_t)\|\right].
\]

For constant probing, \(\delta_t\equiv\delta\), and thus, from
\eqref{eq:two_point_main_bound},
\begin{align*}
\sum_{t\in\evenset{T}}
\mathbb{E}\!\left[\|\nabla \Phi_t(u_t)\|\right]
\le
\frac{\sqrt{2\pi p}}{\delta}
\bigg[
\Phi_0(u_0)-\Phi_{\mathrm{low}}\\
\qquad +
\frac{L_{\nabla\Phi}}{4}
\sum_{t\in\evenset{T}}\delta^2
+2 D_T+2B_T
\bigg].
\end{align*}
Since \(|\evenset{T}|=T/2\), we have
$\sum_{t\in\evenset{T}}\delta^2
=
\frac{T}{2}\delta^2,$
which gives the constant-probing bound stated in
\Cref{prop:instantaneous_local_regret}. Substituting
\eqref{eq:constant_delta} yields the asymptotic bound.

For diminishing probing, \(\delta_t=(t+1)^{-1/2}\). Since
\(t\le T-2\) for all \(t\in\evenset{T}\), we have
$\delta_t\ge \frac{1}{\sqrt{T}}.$
Thus,
\[
\sum_{t\in\evenset{T}}
\mathbb{E}\!\left[\|\nabla \Phi_t(u_t)\|\right]
\le
\sqrt{T}
\sum_{t\in\evenset{T}}
\delta_t
\mathbb{E}\!\left[\|\nabla \Phi_t(u_t)\|\right].
\]
Combining this inequality with \eqref{eq:two_point_main_bound} gives
\begin{align*}
\mathcal{R}_1^{\mathrm{even}}(T)
&\le
G\sqrt{2\pi pT}
\bigg[
\Phi_0(u_0)-\Phi_{\mathrm{low}}\\
&\qquad +
\frac{L_{\nabla\Phi}}{4}
\sum_{t\in\evenset{T}}\delta_t^2
+2 D_T+2B_T
\bigg].
\end{align*}
Finally,
\[
\sum_{t\in\evenset{T}}\delta_t^2
=
\sum_{t\in\evenset{T}}\frac{1}{t+1}
\le
\sum_{t=0}^{T-1}\frac{1}{t+1}
\le
1+\log T,
\]
which proves the diminishing-probing bound stated in
\Cref{prop:instantaneous_local_regret}.
\end{proof}

\Cref{prop:instantaneous_local_regret} links the stationarity guarantees of
this paper with the local-regret framework of~\cite{EH-KS-CZ:17}. In the
special case \(w=1\), the smoothed objective satisfies \(F_{t,1}=\Phi_t\),
making the two notions directly comparable. The proposition shows that the
constant-probing regime yields
\(\mathcal R_1^{\mathrm{even}}(T)=\mathcal O(T)\), matching the scaling
obtained in~\cite{EH-KS-CZ:17} while accommodating the same worst-case drift
growth \(D_T=\mathcal O(T)\). In contrast, the diminishing-probing regime
improves this scaling to
\(\mathcal O(\sqrt{T}\log T)\), up to additional terms depending on
\(D_T\) and \(B_T\). Unlike~\cite{EH-KS-CZ:17}, however, the present analysis
provides guarantees directly on the instantaneous objectives \(\Phi_t\) and
makes the impact of temporal drift and oracle inaccuracies explicit through
\(D_T\) and \(B_T\).

\section{Application to control of dynamical systems}
\label{sec:feedbackOptimization_algorithms}

\subsection{Online optimal equilibrium selection as a time-varying optimization problem}
\label{sec:feedback_opt_description}
A prominent application of the framework in
\Cref{prob:time_varying_inexact_design} is feedback optimization 
\cite{GB-JC-JP-ED:21-tcns,AH-SB-GH-FD:20}, which is a control framework whereby a 
controller is designed to automatically select optimal operating points for a control system using only real-time measurements rather than explicit 
plant models.
Concretely, consider a system to be controlled (hereafter called the \textit{plant}) 
modeled by:
\begin{align} \label{eq:plant}
    x_{t+1} &= f(x_t,u_t,w_t), & 
    y_{t+1} &= h(x_{t+1},w_t),
\end{align}
where $x_t \in \real^n$ is the system state at time $t \in \naturalnneg$, 
$u_t \in \real^p$ is the control input, $y_t \in \real^q$ is the measured 
output, and $w_t \in \real^r$ models a deterministic disturbance. In line with 
standard control frameworks, we require the following. 

\begin{assumption}[\textbf{\textit{Regularity and contractivity of the plant}}]
\label{as:plant_regular}
There exist constants 
$L_{f,x} \in (0,1)$, 
$L_{f,u} \ge 0$, 
$L_{f,w} \ge 0$, 
$L_{h,x} \ge 0$, and 
$L_{h,w} \ge 0$
such that, for all 
$x,\bar x \in \mathbb{R}^n$, 
$u,\bar u \in \mathbb{R}^p$, and
$w,\bar w \in \mathbb{R}^r$,
\begin{align*}
\|f(x,u,w) - f(\bar x,u,w)\|
&\le
L_{f,x} \|x-\bar x\|, & 
\|f(x,u,w) - f(x,\bar u,w)\|
&\le
L_{f,u} \|u-\bar u\|,\\
\|f(x,u,w) - f(x,u,\bar w)\|
&\le
L_{f,w} \|w-\bar w\|, &
\|h(x,w) - h(\bar x,w)\|
&\le
L_{h,x} \|x-\bar x\|,\\
\|h(x,w) - h(x,\bar w)\|
&\le
L_{h,w} \|w-\bar w\|. \tag*{\QEDB}
\end{align*}
\end{assumption}

In words, Assumption~\ref{as:plant_regular} ensures that the plant is contractive in 
the state, that variations in the control input $u$ induce bounded changes in 
the state $x$, and that variations in $x$ lead to bounded changes in the 
output $y$. We note that this is a classical assumption in the existing 
literature~\cite{GB-JC-JP-ED:21-tcns,AH-SB-GH-FD:20}.

Under Assumption~\ref{as:plant_regular}, by the Banach Fixed-Point Theorem, 
for every $(u,w) \in \mathbb{R}^p \times \mathbb{R}^r$ there exists a unique 
equilibrium of~\eqref{eq:plant}; that is, there exists a unique mapping $x_{\mathrm{ss}}: \real^p \times \real^r \to \real^n$ such that, for all $(u,w)$,
\begin{align}\label{eq:x_ss}
f\big(x_{\mathrm{ss}}(u,w),u,w\big) = x_{\mathrm{ss}}(u,w).
\end{align}

Define the input-output steady-state map 
$y_{\mathrm{ss}}: \real^p \times \real^r \to \real^q$ of~\eqref{eq:plant} as:
\begin{equation}\label{eq:ssMap}
    y_{\mathrm{ss}}(u,w) \defeq  h(x_{\mathrm{ss}}(u,w),w).
\end{equation}
The feedback optimization problem consists in devising a control 
algorithm that regulates~\eqref{eq:plant} to the solution of the 
following equilibrium-selection problem:
\begin{align} \label{eq:output_regulation}
        \min_{u \in \real^p, y\in \real^q} ~  \Psi (u,y), \qquad  
        \textrm{subject to:} \quad y = y_{\mathrm{ss}}(u,w_t),
\end{align}
where $\Psi: \real^p \times \real^q \to \real$, termed \textit{terminal cost}, 
is a metric that quantifies plant performance at steady-state. 
Problem~\eqref{eq:output_regulation} seeks a control $u$ and a corresponding 
steady-state 
output $y = y_{\mathrm{ss}}(u,w_t)$ that minimize the performance metric 
$\Psi(u,y)$.
The problem is related with the classical \textit{output regulation problem}
\cite{ED:76}, with the difference that in the latter the output to be tracked 
is given a priori as a reference signal, while in~\eqref{eq:output_regulation} 
the output to be tracked is implicitly specified as the solution of an 
optimization problem (cf. \eqref{eq:output_regulation}).

\begin{remark}[Time-varying nature of~\eqref{eq:output_regulation}]
\label{rem:temp_variab_psi}
Although the loss $\Psi(u,y)$ in \eqref{eq:output_regulation} is 
time-invariant, the optimization problem \eqref{eq:output_regulation} is 
time-varying due to the presence of the time-varying disturbance $w_t$ that 
parametrizes the constraint. Hence, in analogy 
with~\eqref{eq:optimization_main}, we will seek \textit{sequences} of critical 
points.~
\QEDB\end{remark}

By substituting the constraint into the loss, \eqref{eq:output_regulation} is 
an instance of~\eqref{eq:optimization_main} with \textit{reduced cost}
\begin{align}\label{eq:Phi_PhiTilde}
\Phi_t(u) \defeq \Psi \big(u,y_{\mathrm{ss}}(u,w_t)\big).
\end{align}

In line with Assumption~\ref{as:smoothness}, we require the following standard\footnote{A 
sufficient set of conditions ensuring 
Assumption~\ref{as:Psi_Lipschitz} is that
$\Psi(u,y)$ is continuously differentiable in $u$, with $\nabla_u \Psi(u,y)$
Lipschitz in $u$ (uniformly in $y$) and Lipschitz in $y$ (uniformly in $u$), 
and that $y_{\mathrm{ss}}(u,d)$ is continuously differentiable in $u$, with
$\nabla_u y_{\mathrm{ss}}(u,d)$ uniformly bounded and Lipschitz in $u$. 
Under these conditions, $\Phi_t(u)=\Psi(u,y_{\mathrm{ss}}(u,w_t))$ has 
Lipschitz continuous gradient, i.e., Assumption~\ref{as:smoothness} holds.} 
requirement.

\begin{assumption}[Regularity of $\Psi$]
\\
\label{as:Psi_Lipschitz}
The composite function $\Phi_t(u) = \Psi(u,y_{\mathrm{ss}}(u,w))$ satisfies
Assumption~\ref{as:smoothness}. Moreover, the function 
$\Psi(u,y)$ is $L_{\Psi,y}$-Lipschitz continuous in its second argument.~
\QEDB\end{assumption}


The optimal equilibrium-seeking problem~\eqref{eq:output_regulation} is a 
representative setting in which exact evaluations of the reduced 
cost~\eqref{eq:Phi_PhiTilde} are generally not available, following 
\eqref{eq:oracle_inexact}. This is because of two main practical challenges (C):

\begin{enumerate}[label={(C\arabic*)}, ref=C\arabic*, leftmargin=*]
\item \label{C1}
Evaluating~\eqref{eq:Phi_PhiTilde} requires accurate knowledge of 
the plant, in particular, of the steady-state map 
$y_{\mathrm{ss}}(u,w)$ (which, in turn, depends on $f(x,u,w)$ 
and $h(x,w)$--see~\eqref{eq:x_ss}--\eqref{eq:ssMap}). 
Such knowledge is often unrealistic or prohibitive in practice, 
due to the difficulties associated with constructing accurate 
plant models.

\item \label{C2}
Evaluating~\eqref{eq:Phi_PhiTilde} requires access to the 
disturbance $w_t$ at each time $t$. However, this signal is 
typically unknown or unmeasurable in practice (since it models unknown exogenous 
disturbances affecting the plant).
\end{enumerate}
For these reasons, methods that require exact oracle evaluations of the 
reduced cost $\Phi_t(u)$ as in \eqref{eq:Phi_PhiTilde} are inapplicable. 
Instead, one often has access only to zeroth-order oracle evaluations of the 
terminal cost:
\begin{align}\label{eq:oracle_tv_problem_psi}
(u, y) \mapsto \Psi(u,y).
\end{align}

\noindent
We formalize the problem of interest in this section as follows. 

\begin{problem}[Optimal equilibrium selection under bandit feedback]
\label{prob:feedback_design}
Design a control algorithm, relying only on online oracle 
information as in~\eqref{eq:oracle_tv_problem_psi}, to regulate 
the system~\eqref{eq:plant} to solutions of the optimal output 
regulation problem~\eqref{eq:output_regulation}.
\QEDB
\end{problem}

\subsection{Algorithm synthesis}


Since oracle evaluations of the form
\begin{align}\label{eq:oracle_feedback_optimization_exact}
(t,u) \mapsto \Psi \big(u, y_{\mathrm{ss}}(u,w_t)\big)
\end{align}
are not directly accessible in practice (cf. challenges 
\eqref{C1}--\eqref{C2}), we propose to replace the steady-state 
output $y_{\mathrm{ss}}(u,w_t)$ 
in~\eqref{eq:oracle_feedback_optimization_exact} with instantaneous plant 
measurements, leading to an algorithm based on
\emph{online oracle evaluations} of the form
\begin{align}\label{eq:oracle_feedback_optimization}
(t, u_t) \mapsto \Psi(u_t, y_{t+1}), 
&&
y_{t+1} = h\big(f(x_t, u_t, w_t), w_{t}\big).
\end{align}
In other words, $y_{t+1}$ denotes the plant output at time $t+1$ generated by applying the input $u_t$ at time $t$.

A two-point direct-search method based on the oracle
\eqref{eq:oracle_feedback_optimization}, is presented in 
\Cref{alg:two_point_ds_fo}. 
At each iteration, the algorithm applies the \textit{current control input} 
$u_t$ to the plant, measures the plant output, and evaluates the objective 
$\Psi(u_t,y_{t+1})$ (line~\ref{alg:two_point_ds_fo:l1}). 
It then samples a random search 
direction $v_t$ and forms a \textit{candidate control input} $u_{t+1}$ by 
perturbing $u_t$ along this direction with a probing radius $\delta_t$ 
(line~\ref{alg:two_point_ds_fo:l2}). The candidate input is subsequently 
applied to the plant, the plant output is measured, and the objective 
$\Psi(u_{t+1},y_{t+2})$ is evaluated (line~\ref{alg:two_point_ds_fo:l3}). Finally, an accept/reject step is performed: 
the candidate control is accepted if it yields a lower objective value than the 
current control; otherwise, it is rejected (line~\ref{alg:two_point_ds_fo:l4}).

\begin{remark}[Online nature of \Cref{alg:two_point_ds_fo}]
Notice that objective evaluations (lines~\ref{alg:two_point_ds_fo:l1},\ref{alg:two_point_ds_fo:l3}) are 
performed in an \textit{online} fashion, i.e., without restarting or letting 
the plant converge to a steady-state. This is a key difference relative to methods that rely on timescale separation--see, e.g., \cite{GG:26}.~
\QEDB\end{remark}

\begin{algorithm}[t]
\caption{Two-point direct search for equilibrium selection}
\label{alg:two_point_ds_fo}
\begin{algorithmic}[1]
\Require Initial input $u_0 \in \real^p$, initial state $x_0 \in \real^n$, 
probing ratios $\{\delta_t\}_{t \ge 0}$

\State $t \gets 0$
\While{stopping criterion not met}
    \State \label{alg:two_point_ds_fo:l1} Apply $u_t$ to the plant, wait for a  plant update:
    \begin{align*}
        x_{t+1} &= f(x_t,u_t,w_t), &
        y_{t+1} &= h(x_{t+1},w_{t}),
    \end{align*}
    \qquad and evaluate $\Psi(u_t,y_{t+1})$.

    \State \label{alg:two_point_ds_fo:l2} Draw 
    $v_t \sim \mathcal{N}(0,\frac{1}{p}I_p)$ and set 
    $u_{t+1} = u_t+\delta_t v_t$.

    \State \label{alg:two_point_ds_fo:l3} Apply $u_{t+1}$ to the plant, wait for a plant update:
    \begin{align*}
        x_{t+2} &= f(x_{t+1},u_{t+1},w_{t+1}), &
        y_{t+2} &= h(x_{t+2},w_{t+1}),
    \end{align*}
    \qquad and evaluate $\Psi(u_{t+1},y_{t+2})$.

    \State \label{alg:two_point_ds_fo:l4} Update the input as
    $u_{t+2} =
        \begin{cases}
            u_{t+1}, 
            & \text{if } \Psi(u_{t+1},y_{t+2}) \le \Psi(u_t,y_{t+1}),\\
            u_t, 
            & \text{otherwise.}
        \end{cases}$

    \State $t \gets t+2$
\EndWhile
\end{algorithmic}
\end{algorithm}

\subsection{Arbitrary disturbance signals}
\label{sec:tv_d}

In this section, we analyze \Cref{alg:two_point_ds_fo} with $w_t$ an 
arbitrary signal. We begin by observing that the use of the online oracle
\eqref{eq:oracle_feedback_optimization} induces an optimization framework with 
inexact oracle evaluations analogous to that  
in~\eqref{eq:oracle_inexact}. In fact, the formulation 
\eqref{eq:oracle_feedback_optimization} is equivalent to that 
in~\eqref{eq:oracle_inexact} upon letting:
\begin{align}\label{eq:oracle_feedback_optimization_v2}
\tilde{\Phi}_t(u) \defeq \Psi\big(u, h(f(x_{t}, u, w_{t}), w_{t})\big).
\end{align}
Analogously, we will denote the exact oracle by:
\begin{align}\label{eq:exact_oracle_definition}
\Phi_t(u) \defeq \Psi \big(u, y_{\mathrm{ss}}(u,w_t)\big).
\end{align}

\begin{remark}[Time dependence of the inexact oracle]
\label{rem:depandence_t_inexact_oracle}
It is worth noting that \eqref{eq:oracle_feedback_optimization_v2} highlights 
that the time dependence of the inexact oracle $\tilde{\Phi}_t$ arises from 
two sources: (i) the variation of the disturbance sequence $w_t$, and (ii) the 
evolution of the system state $x_t$ induced by the dynamics~\eqref{eq:plant}. 
In contrast, in the exact oracle 
setting~\eqref{eq:oracle_feedback_optimization_exact}, temporal variability 
arises solely from $w_t$ (see 
Remark~\ref{rem:temp_variab_psi}). Hence, interestingly, the use of an inexact oracle 
yields an additional source of temporal variability through the plant's 
transient dynamics.
\QEDB\end{remark}

To capture the essence of the problem in this setting, we define a pointwise
version of the oracle error sequence \eqref{as:inexact_oracle}:
\begin{align}\label{eq:bt_fo}
b_t^{\mathrm{fo}} (u_t)
&\defeq
|\tilde{\Phi}_t(u_t)-\Phi_t(u_t)|, &
B_T^{\mathrm{fo}} &\defeq \sum_{t=0}^{T-1} \mathbb{E}\!\left[b_t^{\mathrm{fo}}(u_t)\right].
\end{align}
Note that, unlike \eqref{as:inexact_oracle}, \(b_t^{\mathrm{fo}}(u_t)\) is a pointwise (decision-dependent) measure of oracle inexactness rather than a worst-case estimate.
We also observe that \(b_t^{\mathrm{fo}}(u_t)\) and \(\mu_t\) quantify the same 
underlying property, namely, the mismatch between the transient plant output and its corresponding steady-state value. In particular, by \eqref{eq:mu_t_bound} and the \(L_{\Psi,y}\)-Lipschitz continuity of \(\Psi\),
\[
b_t^{\mathrm{fo}}(u_t)
\le
L_{\Psi,y}\mu_t.
\]

We begin with the following result, which characterizes the temporal drift 
\eqref{eq:objective_drift} and the oracle inexactness 
\eqref{eq:bt_fo} for the control framework studied in this section.

\begin{lemma}[Feedback-induced oracle error and disturbance-induced drift]
\label{lem:feedback_oracle_error} 
Let Assumptions~\ref{as:plant_regular}
and~\ref{as:Psi_Lipschitz} hold. Then, for every
$t\in\naturalnneg$, the temporal drift satisfies \eqref{eq:objective_drift}
with
\begin{align}\label{eq:feedback_drift_bound}
d_t
&\le
L_{\Psi,y}L_{y_{\mathrm{ss}},w}
\|w_{t+1}-w_t\|,
\end{align}
where 
$L_{y_{\mathrm{ss}},w} \defeq L_{h,w}
+
\frac{L_{h,x}L_{f,w}}{1-L_{f,x}}.$
Moreover, the online oracle $\tilde \Phi_t$ satisfies 
\eqref{eq:bt_fo} with
\begin{align}\label{eq:feedback_inexactness_bound}
b_t^{\mathrm{fo}}(u_t)\leq L_{\Psi,y} 
\|y_{t+1}-y_{\mathrm{ss}}(u_t,w_t)\|,
\end{align}
where $y_{t+1} = h\big(f(x_t, u_t, w_t), w_t\big).$
\QEDB\end{lemma}

\begin{proof}
We first bound the temporal drift. By definition,
\begin{align*}
d_t
&= \sup_{u \in \real^p}
\left|
\Psi\bigl(u,y_{\mathrm{ss}}(u,w_{t+1})\bigr)
-
\Psi\bigl(u,y_{\mathrm{ss}}(u,w_t)\bigr)
\right|.
\end{align*}
Since $\Psi$ is $L_{\Psi,y}$-Lipschitz continuous in its second argument
(Assumption~\ref{as:Psi_Lipschitz}), it follows that
\begin{align*}
d_t
&\le
L_{\Psi,y} \sup_{u \in \real^p}
\left\|
y_{\mathrm{ss}}(u,w_{t+1})
-
y_{\mathrm{ss}}(u,w_t)
\right\|.
\end{align*}
By Assumption~\ref{as:plant_regular}, the steady-state map satisfies
\[
\left\|
x_{\mathrm{ss}}(u,w_{t+1})
-
x_{\mathrm{ss}}(u,w_t)
\right\|
\le
\frac{L_{f,w}}{1-L_{f,x}}
\|w_{t+1}-w_t\|.
\]
Using again Assumption~\ref{as:plant_regular}, we obtain
\begin{align*}
&
\left\|
y_{\mathrm{ss}}(u,w_{t+1})
-
y_{\mathrm{ss}}(u,w_t)
\right\|
\\
&\qquad\le
L_{h,x}
\left\|
x_{\mathrm{ss}}(u,w_{t+1})
-
x_{\mathrm{ss}}(u,w_t)
\right\|
+
L_{h,w}\|w_{t+1}-w_t\| \\
&\qquad\le
\left(
L_{h,w}
+
\frac{L_{h,x}L_{f,w}}{1-L_{f,x}}
\right)
\|w_{t+1}-w_t\|.
\end{align*}
Combining the previous inequalities and using the definition of
$L_{y_{\mathrm{ss}},w}$ yields \eqref{eq:feedback_drift_bound}.

We next prove the inexact-oracle bound. By definition,
\[
\Phi_t(u_t)=\Psi(u_t,y_{\mathrm{ss}}(u_t,w_t)),
\qquad
\tilde{\Phi}_t(u_t)=\Psi(u_t,y_{t+1}),
\]
where $y_{t+1} = h\big(f(x_t, u_t, w_t), w_t\big).$
Using again the Lipschitz continuity of $\Psi$ in its second argument,
\begin{align*}
|\tilde{\Phi}_t(u_t)-\Phi_t(u_t)|
&\le
L_{\Psi,y}
\|y_{t+1}-y_{\mathrm{ss}}(u_t,w_t)\| 
\end{align*}
Therefore, the trajectory-dependent feedback oracle error satisfies
\[
b_t^{\mathrm{fo}}(u_t)
\le
L_{\Psi,y}
\|y_{t+1}-y_{\mathrm{ss}}(u_t,w_t)\|,
\]
which proves \eqref{eq:feedback_inexactness_bound}.
\end{proof}

\Cref{lem:feedback_oracle_error} shows that the variability of the reduced
objective $\Phi_t$ is induced by the disturbance variation
$\|w_{t+1}-w_t\|$, while the oracle inexactness is determined by the mismatch
between the measured output $y_{t+1}$ and the steady-state 
$y_{\mathrm{ss}}(u_t,w_t)$.

\Cref{lem:feedback_oracle_error} provides a key insight: when the plant 
operates near equilibrium, \(y_{t+1}\approx y_{\mathrm{ss}}(u_t,w_t)\), implying 
that the oracle inexactness remains small 
(see \eqref{eq:feedback_inexactness_bound}). Motivated 
by this observation, we impose the following assumption.

\begin{assumption}[Uniform bound on transient steady-state output mismatch]
\label{as:mu_t_bound} 
There exists a sequence $\{\mu_t\}_{t \in \naturalnneg}$ such that, for every
$t\in\naturalnneg$,
\begin{align}\label{eq:mu_t_bound}
\|
y_{t+1}-y_{\mathrm{ss}}(u_t,w_t)
\|
\le
\mu_t,
\end{align}
where
$y_{t+1}
=
h\bigl(f(x_t,u_t,w_t),w_t\bigr).$
\QEDB
\end{assumption}

Assumption~\ref{as:mu_t_bound} quantifies the mismatch between the instantaneous output
generated by the plant dynamics and the steady-state output, associated with the
input--disturbance pair $(u_t,w_t)$. 
We stress that, $\mu_t$ is allowed to depend on $t$  because $w_t$ is 
time-varying.

A direct application of \Cref{thm:two_point_tv_inexact} to the feedback 
optimization framework yields the following result, which  constitutes our main 
result for \Cref{alg:two_point_ds_fo}.

\begin{remark}
    The conclusions of Theorems~\ref{thm:two_point_tv_inexact} and \ref{thm:two_point_diminishing} also hold when \(B_T\) is
replaced by the cumulative oracle error evaluated at the points queried
by the algorithm (\(B_T^{\mathrm{fo}}\)). Indeed, the proofs only use the oracle errors at \(u_t\) and
\(u_{t+1}\).
\end{remark}

\begin{corollary}[Stationarity guarantees of \Cref{alg:two_point_ds_fo}]
\label{cor:two_point_fo}
Suppose that
Assumptions~\ref{as:plant_regular}, \ref{as:Psi_Lipschitz}, and~\ref{as:mu_t_bound} hold. Then, for every
$T \in \Neventwo$, the iterates generated by
\Cref{alg:two_point_ds_fo} satisfy
\begin{align}
\frac{1}{\sqrt{2\pi p}}
\sum_{t \in \evenset{T}}
\delta_t
\mathbb{E}\!\left[\|\nabla \Phi_t(u_t)\|\right]
&\le
\Phi_0(u_0)-\Phi_{\mathrm{low}}
+
\frac{L_{\nabla \Phi}}{4}
\sum_{t \in \evenset{T}} \delta_t^2
\notag\\
&\qquad+
2L_{\Psi,y}L_{y_{\mathrm{ss}},w}
\sum_{t=0}^{T-1}\|w_{t+1}-w_t\|
+
2L_{\Psi,y}
\sum_{t=0}^{T-1}\mathbb{E}[\mu_t] .
\label{eq:two_point_fo_bound}
\end{align}
Consequently,
\begin{align}
\min_{t \in \evenset{T}}
\mathbb{E}\!\left[\|\nabla \Phi_t(u_t)\|\right]
& \le
\frac{\sqrt{2\pi p}}
{\sum_{t \in \evenset{T}}\delta_t}
\Bigg[
\Phi_0(u_0)-\Phi_{\mathrm{low}}
+
\frac{L_{\nabla \Phi}}{4}
\sum_{t \in \evenset{T}} \delta_t^2
\notag\\
&\qquad+
2L_{\Psi,y}L_{y_{\mathrm{ss}},w}
\sum_{t=0}^{T-1}\|w_{t+1}-w_t\|
+
2L_{\Psi,y}
\sum_{t=0}^{T-1}\mathbb{E}[\mu_t]
\Bigg].
\label{eq:two_point_fo_cor}
\end{align}
\QEDB
\end{corollary}

\begin{proof}
By \Cref{lem:feedback_oracle_error} and Assumption~\ref{as:mu_t_bound}, the cumulative
drift and oracle-error terms satisfy
\begin{align*}
D_T
&\le
L_{\Psi,y}L_{y_{\mathrm{ss}},w}
\sum_{t=0}^{T-1}\|w_{t+1}-w_t\|,
\\
B_T^{\mathrm{fo}}
&\le
L_{\Psi,y}
\sum_{t=0}^{T-1}\mathbb{E}[\mu_t].
\end{align*}
Applying the trajectory-dependent version of \Cref{thm:two_point_tv_inexact}  with
\(B_T\) replaced by \(B_T^{\mathrm{fo}}\) yields the claim.
\end{proof}

\Cref{cor:two_point_fo} shows that the convergence properties of
\Cref{alg:two_point_ds_fo} are jointly affected by four terms: the descent
induced by the optimization algorithm, the probing ratio, the temporal 
variability of $w_t$, and the mismatch between the transient and steady-state 
plant output. 

The following result characterizes the iteration complexity of 
\Cref{alg:two_point_ds_fo} under diminishing probing ratios.

\begin{corollary}[Iteration complexity of \Cref{alg:two_point_ds_fo}--diminishing probing ratio]
\label{cor:two_point_fo_diminishing}
Suppose that
Assumptions~\ref{as:plant_regular}, \ref{as:Psi_Lipschitz}, and~\ref{as:mu_t_bound} hold.
Fix $\varepsilon > 0$, and let the probing ratios be
\[
\delta_t = 1/\sqrt{t+1}.
\]
Suppose that the cumulative variability satisfies
\begin{align}\label{eq:low_bound_eps_two_point_fo_diminishing}
\scalebox{0.9}{$\displaystyle
2L_{\Psi,y}L_{y_{\mathrm{ss}},w}
\sum_{t=0}^{T-1}\|w_{t+1}-w_t\|
+
2L_{\Psi,y}\sum_{t=0}^{T-1}\mathbb{E}[\mu_t]
\le
\frac{\varepsilon\sqrt{T}}
{2(\sqrt{2}+1)\sqrt{2\pi p}}
$}
\end{align}
Then, for any $T \in \Neventwo$ satisfying
\begin{align}\label{eq:lower_T_fo_diminishing}
T
\ge
& \max\!\Bigg\{
\frac{2\pi p(\sqrt{2}+1)^2}{\varepsilon^2}
\bigl(4(\Phi_0(u_0)-\Phi_{\mathrm{low}})+L_{\nabla \Phi}\bigr)^2,
\notag\\
&
\frac{8\pi p(\sqrt{2}+1)^2L_{\nabla \Phi}^2}{\varepsilon^2}
\log^2\!\Bigl(
\frac{2(\sqrt{2}+1)\sqrt{2\pi p}\,L_{\nabla \Phi}}
{\varepsilon}
\Bigr)
\Bigg\},
\end{align}
the iterates generated by \Cref{alg:two_point_ds_fo} satisfy
\[
\min_{t \in \evenset{T}}
\mathbb{E}\!\left[\|\nabla \Phi_t(u_t)\|\right]
\le
\varepsilon.
\]

\end{corollary}

\begin{proof}
By \Cref{lem:feedback_oracle_error} and
Assumption~\ref{as:mu_t_bound}, the temporal drift and oracle error satisfy
\[
D_T
\le
L_{\Psi,y}L_{y_{\mathrm{ss}},w}
\sum_{t=0}^{T-1}\|w_{t+1}-w_t\|,
\]
and
\[
B_T^{\mathrm{fo}}
\le
L_{\Psi,y}\sum_{t=0}^{T-1}\mathbb{E}[\mu_t].
\]
Therefore,
\[
2 D_T +2B_T^{\mathrm{fo}}
\le
2L_{\Psi,y}L_{y_{\mathrm{ss}},w}
\sum_{t=0}^{T-1}\|w_{t+1}-w_t\|
+
2L_{\Psi,y}\sum_{t=0}^{T-1}\mathbb{E}[\mu_t].
\]
Hence, the stated condition implies the trajectory-dependent counterpart
of \eqref{eq:E_condition_diminishing}, and the claim follows from the trajectory-dependent
version of \Cref{thm:two_point_diminishing}.
\end{proof}

\Cref{cor:two_point_fo_diminishing} shows that, under the stated conditions, 
\Cref{alg:two_point_ds_fo} reaches an $\varepsilon$-stationary point in finite time. More precisely, the result establishes that \Cref{alg:two_point_ds_fo} 
achieves the same $O(p\varepsilon^{-2}\log^2(\varepsilon^{-1}))$ iteration 
complexity as in the time-varying setting. 
In analogy with \Cref{thm:two_point_const,thm:two_point_diminishing}, the claim 
requires the combined effect of temporal drift of $w_t$ and transient-steady 
state output mismatch to remain sufficiently small (cf. 
\eqref{eq:low_bound_eps_two_point_fo_diminishing}).

\subsection{Special case: feedback optimization under constant disturbances}
\label{sec:const_d}

In this section, we refine the results of \Cref{sec:tv_d} to the following 
setting. 

\begin{assumption}[Constant disturbance]\label{as:constant_disturbance}
The disturbance sequence $\{w_t\}_{t \in \naturalnneg}$ in~\eqref{eq:plant} 
satisfies $w_t = w$ for all $t \in \naturalnneg$.~
\QEDB
\end{assumption}

This case is of particular interest in output regulation \cite{ED:76}; note that 
several methods for feedback optimization have been proposed in this setting 
\cite{GB-JC-JP-ED:21-tcns,AH-SB-GH-FD:20}.


The following result shows that, under constant disturbances and diminishing
probing, the feedback-induced oracle error
$b_t^{\mathrm{fo}}(u_t)$ decays geometrically up to a vanishing
$O(\delta_t)$ term.

\begin{lemma}[Probing-ratio-dependent oracle error]
\label{lem:stepsize_dependent_bt}
\\
Suppose that
Assumptions~\ref{as:plant_regular}, \ref{as:Psi_Lipschitz}, and~\ref{as:constant_disturbance} hold.
Consider the iterates generated by \Cref{alg:two_point_ds_fo} with diminishing
probing ratios $\delta_t = \frac{1}{\sqrt{t+1}}$ for $t\in\Nevennneg$.
Then there exist constants $\bar c_0,\bar c_b\ge 0$ and $\rho\in(0,1)$ such
that, for all $t\in\Nevennneg$,
\begin{equation}
\label{eq:bt_decay_lemma}
\mathbb{E}\!\left[b_t^{\mathrm{fo}}(u_t)\right]
+
\mathbb{E}\!\left[b_{t+1}^{\mathrm{fo}}(u_{t+1})\right]
\le
\bar c_0 \rho^{t/2}
+
\bar c_b \delta_t
\end{equation}
In particular,
$\mathbb{E}[b_t^{\mathrm{fo}}(u_t)]=O(\delta_t)$ up to an exponentially
decaying transient term, and
\begin{equation}
\label{eq:sum_bt_decay_lemma}
B_T^{\mathrm{fo}}
=
O(\sqrt{T}).
\end{equation}
\end{lemma}

\begin{proof}
Let
\[
e_t \defeq \|x_t-x_{\mathrm{ss}}(u_t,w)\|,
\qquad
L_{x_{\mathrm{ss}},u}\defeq \frac{L_{f,u}}{1-L_{f,x}}.
\]
Since $w_t\equiv w$, the steady-state map is time-invariant. Moreover, by
Assumption~\ref{as:plant_regular}, the map $x_{\mathrm{ss}}(\cdot,w)$ is
$L_{x_{\mathrm{ss}},u}$-Lipschitz in its first argument.

We first derive a recursion for $e_t$ over even indices. Since
$x_{t+1}=f(x_t,u_t,w)$ and
$x_{\mathrm{ss}}(u_t,w)=f(x_{\mathrm{ss}}(u_t,w),u_t,w)$, contractivity gives
\[
\|x_{t+1}-x_{\mathrm{ss}}(u_t,w)\|
\le
L_{f,x}e_t.
\]
Furthermore, using
$x_{t+2}=f(x_{t+1},u_{t+1},w)$, adding and subtracting
$x_{\mathrm{ss}}(u_t,w)$, and applying Assumption~\ref{as:plant_regular}, we obtain
\begin{align*}
e_{t+2}
&=
\|x_{t+2}-x_{\mathrm{ss}}(u_{t+2},w)\| \\
&\le
\|x_{t+2}-x_{\mathrm{ss}}(u_t,w)\|
+
\|x_{\mathrm{ss}}(u_t,w)-x_{\mathrm{ss}}(u_{t+2},w)\| \\
&\le
L_{f,x}\|x_{t+1}-x_{\mathrm{ss}}(u_t,w)\|
+
L_{f,u}\|u_{t+1}-u_t\| \\
&\quad
+
L_{x_{\mathrm{ss}},u}\|u_{t+2}-u_t\| \\
&\le
L_{f,x}^2 e_t
+
\left(L_{f,u}+L_{x_{\mathrm{ss}},u}\right)\delta_t\|v_t\|.
\end{align*}
In the last step, we used the construction of \Cref{alg:two_point_ds_fo}, which
implies
$\|u_{t+1}-u_t\|=\delta_t\|v_t\|$ and
$\|u_{t+2}-u_t\|\le\delta_t\|v_t\|$.
Let $C_u\defeq L_{f,u}+L_{x_{\mathrm{ss}},u}$ and
$\rho\defeq L_{f,x}^2\in(0,1)$. Taking expectations and using
$\mathbb{E}[\|v_t\|]\le 1$ yields
\[
\mathbb{E}[e_{t+2}]
\le
\rho\,\mathbb{E}[e_t]
+
C_u\delta_t .
\]
Iterating this recursion over even indices gives, for every
$t\in\Nevennneg$,
\[
\mathbb{E}[e_t]
\le
\rho^{t/2}e_0
+
C_u
\sum_{k\in\evenset{t}}
\rho^{(t-2-k)/2}\delta_k .
\]
Since $\delta_t=(t+1)^{-1/2}$ is nonincreasing and the weights are geometric,
there exist constants $c_e,c_\delta\ge 0$ such that
\[
\mathbb{E}[e_t]
\le
c_e\rho^{t/2}
+
c_\delta\delta_t .
\]

We now translate this state mismatch bound into an oracle-error bound. By
definition of the online and reduced costs,
\[
b_t^{\mathrm{fo}}(u_t)
=
|\tilde{\Phi}_t(u_t)-\Phi_t(u_t)|
=
\left|
\Psi(u_t,y_{t+1})
-
\Psi(u_t,y_{\mathrm{ss}}(u_t,w))
\right|.
\]
Using Assumption~\ref{as:Psi_Lipschitz} and then Assumption~\ref{as:plant_regular},
\begin{align*}
b_t^{\mathrm{fo}}(u_t)
&\le
L_{\Psi,y}
\|y_{t+1}-y_{\mathrm{ss}}(u_t,w)\| \\
&\le
L_{\Psi,y}L_{h,x}
\|x_{t+1}-x_{\mathrm{ss}}(u_t,w)\| \\
&\le
L_{\Psi,y}L_{h,x}L_{f,x}e_t .
\end{align*}
The same reasoning applies to the intermediate trial query $u_{t+1}$.
Indeed, since
\[
\|u_{t+1}-u_t\|=\delta_t\|v_t\|,
\]
and the steady-state map is Lipschitz continuous with respect to the input,
the additional steady-state displacement is also of order $\delta_t$.
Taking expectations and using the previous bound on $\mathbb{E}[e_t]$, we
obtain
\[
\mathbb{E}\!\left[b_t^{\mathrm{fo}}(u_t)\right]
+
\mathbb{E}\!\left[b_{t+1}^{\mathrm{fo}}(u_{t+1})\right]
\le
\bar c_0 \rho^{t/2}
+
\bar c_b \delta_t,
\qquad
t\in\mathbb{N}_{\ge 0}^{\mathrm{even}}.
\]

for suitable constants $\bar c_0,\bar c_b\ge 0$. This proves
\eqref{eq:bt_decay_lemma}. Finally, summing over $t=0,\ldots,T-1$ gives
\[
B_T^{\mathrm{fo}}
=
\sum_{t=0}^{T-1}
\mathbb{E}\!\left[b_t^{\mathrm{fo}}(u_t)\right]
=
O(\sqrt{T}),
\]
because $\sum_{t=0}^{T-1}\rho^{t/2}=O(1)$ and
$\sum_{t=0}^{T-1}\delta_t=O(\sqrt{T})$.
\end{proof}

\Cref{lem:stepsize_dependent_bt} shows that, under constant disturbances, the
feedback-induced oracle error decreases proportionally to the probing ratio
$\delta_t$, up to an exponentially decaying transient term. In particular, the
use of diminishing probing ratios causes the transient mismatch between the
plant trajectory and the steady-state operating point to vanish asymptotically,
thereby making the online oracle progressively more accurate over time.
Moreover, by~\eqref{eq:sum_bt_decay_lemma}, the cumulative oracle error grows 
at most on the order of $\sqrt{T}$. \Cref{fig:mu_plot} illustrates the 
statement of \Cref{lem:stepsize_dependent_bt}.

\begin{figure}[t!]
    \centering
    \includegraphics[width=  0.6\columnwidth]{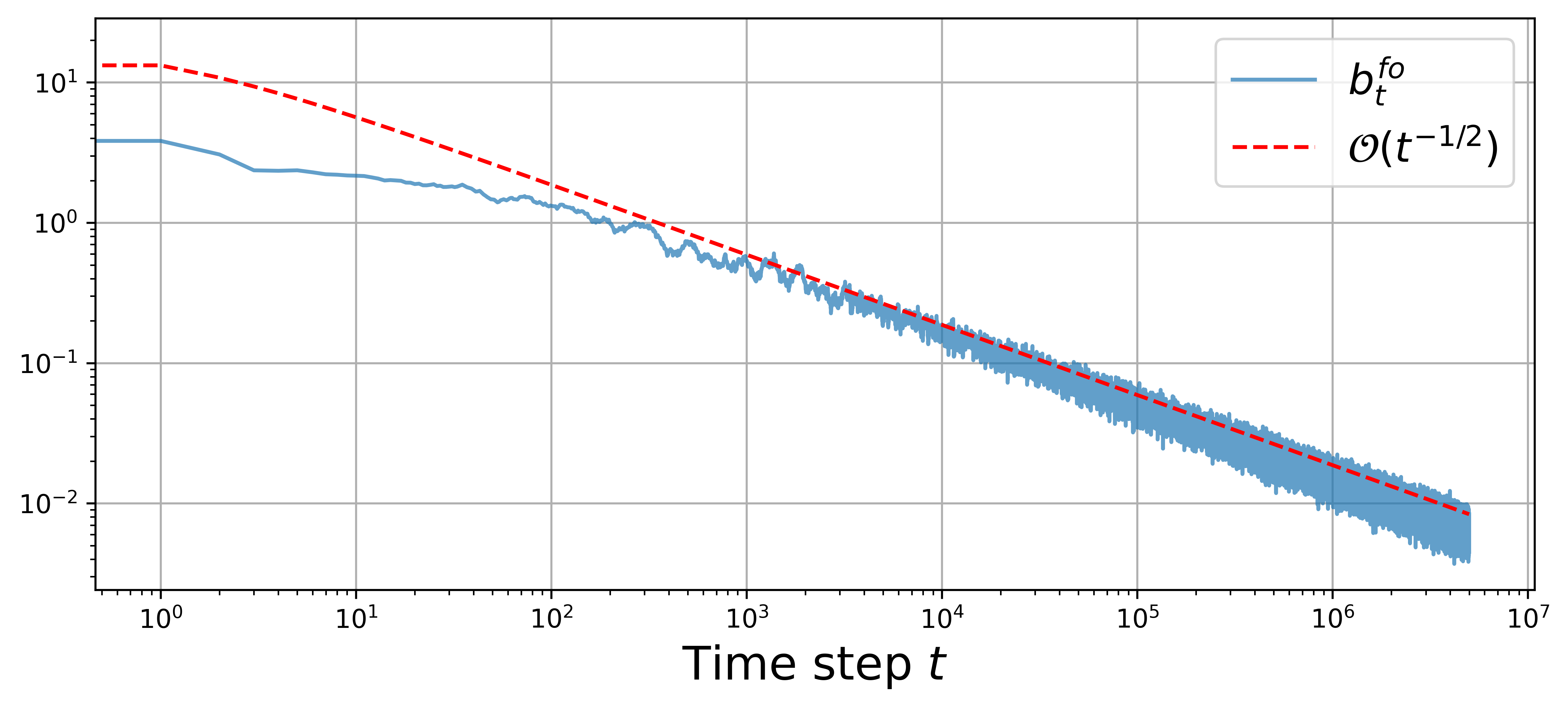}
    \caption{Comparison between the empirical oracle error $b_t^{\mathrm{fo}}(u_t)$ and the theoretical decay rate $O(t^{-1/2})$ predicted by 
    \Cref{lem:stepsize_dependent_bt} in the regime: (i) diminishing probing ratios $\delta_t = 1/\sqrt{t+1}$, and (ii) constant disturbances (cf. Assumption~\ref{as:constant_disturbance}). The numerics illustrate the two regimes captured by \eqref{eq:bt_decay_lemma}: a transient phase followed by an asymptotic regime in which the oracle error scales proportionally to the probing ratio. See \Cref{sec:simulation_results} for details on the numerical values used.}
    \label{fig:mu_plot}
\end{figure}

The conclusion in \Cref{lem:stepsize_dependent_bt} is remarkable, as it 
guarantees that the cumulative oracle error scales as $O(\sqrt{T})$; therefore, 
the assumptions of \Cref{cor:two_point_fo_diminishing} are met, and we can
prove that the algorithm reaches a neighborhood of stationarity.

\begin{theorem}[Iteration complexity of \Cref{alg:two_point_ds_fo} under constant disturbances]
\label{thm:two_point_fo_diminishing_constant_d}
Suppose that
Assumptions~\ref{as:plant_regular}, \ref{as:Psi_Lipschitz}, and~\ref{as:constant_disturbance} hold.
Consider the sequence of diminishing probing ratios
\[
\delta_t = 1/\sqrt{t+1},
\qquad
t \in \Nevennneg.
\]
Let $\bar c_0,\bar c_b,\rho$ be as in \Cref{lem:stepsize_dependent_bt},
and fix $\varepsilon$ such that
\begin{equation}
\label{eq:b-gain-condition}
\varepsilon \geq 
16 \bar c_b (\sqrt{2}+1)\sqrt{2\pi p}.
\end{equation}
Then, for any $T \in \Neventwo$ satisfying
\begin{align}
T
\ge
&
\max\!\Bigg\{
\frac{2\pi p(\sqrt{2}+1)^2}{\varepsilon^2}
\bigl(
4(\Phi(u_0)-\Phi_{\mathrm{low}})
+
L_{\nabla \Phi}
\bigr)^2,
\notag\\
&
\frac{
8\pi p(\sqrt{2}+1)^2L_{\nabla \Phi}^2
}{\varepsilon^2}
\log^2\!\Bigl(
\frac{
2(\sqrt{2}+1)\sqrt{2\pi p}\,L_{\nabla \Phi}
}{\varepsilon}
\Bigr)
\Bigg\},
\label{eq:T-condition-fo-constant-d}
\\
T
\ge
&
\left(
\frac{
8(\sqrt{2}+1)\sqrt{2\pi p}\,\bar c_0
}{
(1-\sqrt{\rho})\varepsilon
}
\right)^2,
\label{eq:T-condition-transient-b}
\end{align}
the iterates generated by \Cref{alg:two_point_ds_fo} satisfy
\[
\min_{t \in \evenset{T}}
\mathbb{E}\!\left[
\|\nabla \Phi(u_t)\|
\right]
\le
\varepsilon.
\]
\end{theorem}

\begin{proof}
Since $w_t \equiv w$, the steady-state objective is time-invariant, i.e.,
$\Phi_t \equiv \Phi$, and therefore $D_T = 0$.

By \Cref{cor:two_point_fo_diminishing}, it suffices to bound the cumulative
oracle error term
\[
2
\sum_{t=0}^{T-1}
\mathbb{E}\!\left[
b_t^{\mathrm{fo}}(u_t)
\right].
\]
By \Cref{lem:stepsize_dependent_bt},
\[
\sum_{t=0}^{T-1} \mathbb E[b_t^{\rm fo}(u_t)]
=
\sum_{t\in[T]_2}
\mathbb E\!\left[
b_t^{\rm fo}(u_t)+b_{t+1}^{\rm fo}(u_{t+1})
\right]
\]
\[
\quad \qquad \le
\sum_{t\in[T]_2}
(\bar c_0\rho^{t/2}+\bar c_b\delta_t).
\]
Using
$\sum_{t=0}^{T-1}\rho^{t/2}
\le
\frac{1}{1-\sqrt{\rho}},$
$\sum_{t=0}^{T-1}\delta_t
\le
2\sqrt{T},$
we obtain
\[
2
\sum_{t=0}^{T-1}
\mathbb{E}\!\left[
b_t^{\mathrm{fo}}(u_t)
\right]
\le
\frac{2\bar c_0}{1-\sqrt{\rho}}
+
4\bar c_b\sqrt{T}.
\]

By \eqref{eq:T-condition-transient-b},
\[
\frac{2\bar c_0}{1-\sqrt{\rho}}
\le
\frac{
\varepsilon\sqrt{T}
}{
4(\sqrt{2}+1)\sqrt{2\pi p}
},
\]
while \eqref{eq:b-gain-condition} implies
\[
4\bar c_b\sqrt{T}
\le
\frac{
\varepsilon\sqrt{T}
}{
4(\sqrt{2}+1)\sqrt{2\pi p}
}.
\]
Combining the two bounds gives
\[
2
\sum_{t=0}^{T-1}
\mathbb{E}\!\left[
b_t^{\mathrm{fo}}(u_t)
\right]
\le
\frac{
\varepsilon\sqrt{T}
}{
2(\sqrt{2}+1)\sqrt{2\pi p}
}.
\]

The remaining initial-gap and smoothness terms are controlled by
\eqref{eq:T-condition-fo-constant-d}, exactly as in
\Cref{cor:two_point_fo_diminishing}. Therefore,
\[
\min_{t \in \evenset{T}}
\mathbb{E}\!\left[
\|\nabla \Phi(u_t)\|
\right]
\le
\varepsilon,
\]
which proves the claim.
\end{proof}


\Cref{thm:two_point_fo_diminishing_constant_d} shows that, under constant
disturbances, \Cref{alg:two_point_ds_fo} achieves the same asymptotic
iteration complexity as the ideal direct-search scheme, despite relying on
transient plant measurements instead of exact steady-state evaluations.

The requirement~\eqref{eq:b-gain-condition} quantifies the intrinsic
accuracy limitation induced by the plant dynamics. Since the residual oracle
error decays proportionally to the probing ratio, the coefficient $\bar c_b$
sets a lower limit on the achievable stationarity tolerance: smaller values of
$\varepsilon$ require correspondingly smaller residual mismatch gain
$\bar c_b$.

Relative to the previous results, the theorem introduces the additional lower
bound~\eqref{eq:T-condition-transient-b} on $T$, which is used to ensure that the
exponentially decaying transient component of the oracle error becomes
negligible relative to the target accuracy~$\varepsilon$.

\section{Heuristic extensions}
\label{sec:extensions}

In this section, we discuss several extensions of the proposed method. Although 
their theoretical analysis is beyond the scope of this paper, these extensions 
are expected to be of significant practical relevance.

\subsection{One-point direct search with residual}

Since  our method requires two oracle evaluations per iteration, it is natural to 
seek a scheme that reuses information from past iterations. Motivated by this 
idea, \Cref{alg:one_point_ds_residual} proposes a one-point direct-search 
method. 

At each iteration, the algorithm evaluates the oracle at a single perturbed 
point $u_t^+=u_t+\delta_t v_t$, where $v_t$ is a Gaussian random direction 
(lines~\ref{alg:one_point_ds_residual:l1} and 
\ref{alg:one_point_ds_residual:l3}). The candidate point is accepted only if 
it improves upon the previously probed value, namely if
$\tilde \Phi_{t+1}(u_t^+) \le \tilde \Phi_t(u_{t-1}^+),$
as checked in line~\ref{alg:one_point_ds_residual:l2}. Otherwise, the current 
iterate is retained. The method therefore requires only one oracle query per 
iteration while leveraging past evaluations for comparisons.

Unlike the two-point methods developed in \Cref{sec:tv_algorithms}, the analysis 
of \Cref{alg:one_point_ds_residual} must explicitly account for temporal drifts 
in the oracle between consecutive iterations. We conjecture that these drifts 
introduce an additional nonvanishing term in bounds of the form 
\eqref{eq:two_point_main_bound}, potentially leading to reduced asymptotic 
resolution relative to the two-point method. 
We leave such an analysis as the scope of future works.

\begin{algorithm}[t]
\caption{Online one-point residual direct search}
\label{alg:one_point_ds_residual}
\begin{algorithmic}[1]
\Require Initial iterate $u_0 \in \real^p$, probing ratios 
$\{\delta_t\}_{t \ge 0}$, oracle 
$\tilde \Phi_t : \real^p \to \mathbb{R}$

\State $t \gets 0$

\State Draw 
$v_0 \sim \mathcal{N}(0,\frac{1}{p}I_p)$ and define
$u_0^+ := u_0 + \delta_0 v_0$

\While{stopping criterion not met}

    \State \label{alg:one_point_ds_residual:l1}
    Evaluate 
    $\tilde \Phi_{t+1}(u_t^+)$

    \State \label{alg:one_point_ds_residual:l2}
    Update the decision as 
    $    u_{t+1} =
        \begin{cases}
            u_t^+, 
            & \text{if } 
            \tilde \Phi_{t+1}(u_t^+) 
            \le 
            \tilde \Phi_t(u_{t-1}^+), \\
            u_t, 
            & \text{otherwise.}
        \end{cases}$

    \State \label{alg:one_point_ds_residual:l3}
    Draw 
    $v_{t+1} \sim \mathcal{N}(0,\frac{1}{p}I_p)$ and define
    $u_{t+1}^+ := u_{t+1} + \delta_{t+1} v_{t+1}$

    \State \label{alg:one_point_ds_residual:l4}
    $t \gets t+1$

\EndWhile
\end{algorithmic}
\end{algorithm}

\subsection{Three-point method}

Inspired from the framework of~\cite{EB-EG-PR:20}, 
\Cref{alg:three_point_ds_tv} proposes a stochastic three-point method for online 
optimization. 

At each iteration, the algorithm first evaluates the oracle at the current  
iterate (line~\ref{alg:three_point_ds_tv:l1}). It then generates two opposite perturbations along the same Gaussian direction
(lines~\ref{alg:three_point_ds_tv:l2} and 
\ref{alg:three_point_ds_tv:l4}), and evaluates the oracle at the corresponding 
candidate points (lines~\ref{alg:three_point_ds_tv:l3} and 
\ref{alg:three_point_ds_tv:l5}).
The update step in line~\ref{alg:three_point_ds_tv:l6} selects the point among 
$\{u_t,u_{t+1},u_{t+2}\}$ with the smallest oracle value. If neither candidate 
improves upon the current iterate, the method retains $u_t$. The algorithm then 
advances by three steps (line~\ref{alg:three_point_ds_tv:l7}).

Three-point methods harness symmetry from the two exploration directions to 
improve the descent~\cite{EB-EG-PR:20}. We conjecture that this 
additional exploration yields at most a constant-factor improvement in iteration 
complexity relative to the two-point method, but leave a formal analysis for 
future~work.

\begin{algorithm}[t]
\caption{Online three-point direct search}
\label{alg:three_point_ds_tv}
\begin{algorithmic}[1]
\Require Initial iterate $u_0 \in \real^p$, probing ratios 
$\{\delta_t\}_{t \ge 0}$, oracle 
$\tilde \Phi_t : \real^p \to \mathbb{R}$

\State $t \gets 0$

\While{stopping criterion not met}

    \State \label{alg:three_point_ds_tv:l1}
    Evaluate $\tilde \Phi_t(u_t)$

    \State \label{alg:three_point_ds_tv:l2}
    Draw 
    $v_t \sim \mathcal{N}(0,\frac{1}{p}I_p)$ and define
    $u_{t+1} := u_t + \delta_t v_t$

    \State \label{alg:three_point_ds_tv:l3}
    Evaluate 
    $\tilde \Phi_{t+1}(u_{t+1})$

    \State \label{alg:three_point_ds_tv:l4}
    Define $u_{t+2} := u_t - \delta_{t} v_{t}$

    \State \label{alg:three_point_ds_tv:l5}
    Evaluate 
    $\tilde \Phi_{t+2}(u_{t+2})$

    \State \label{alg:three_point_ds_tv:l6}
    Update :\hspace{-1.5cm}
    \scalebox{.9}{\parbox{\linewidth}{
    \begin{align*}
        u_{t+3} =
        \begin{cases}
        u_{t+1}, 
        & \text{if }
        \tilde \Phi_{t+1}(u_{t+1})
        \le
        \min\{
        \tilde \Phi_t(u_t),
        \tilde \Phi_{t+2}(u_{t+2})
        \},
        \\
        u_{t+2}, 
        & \text{if }
        \tilde \Phi_{t+2}(u_{t+2})
        \le
        \min\{
        \tilde \Phi_t(u_t),
        \tilde \Phi_{t+1}(u_{t+1})
        \},
        \\
        u_t,
        & \text{otherwise.}
        \end{cases}
    \end{align*}
    }}

    \State \label{alg:three_point_ds_tv:l7}
    $t \gets t+3$

\EndWhile
\end{algorithmic}
\end{algorithm}

\section{Simulation results}
\label{sec:simulation_results}

We test our methods on an instance of the optimal equilibrium-selection 
problem \eqref{eq:output_regulation}, with:
\begin{align}\label{eq:loss_sims_1}
\Psi(u,y)
=
u^\top R_1 u
+
R_2^\top u
+
\gamma\|y\|^2,
\end{align}
where $R_1$ is symmetric positive definite and $\gamma>0$. The objective 
\eqref{eq:loss_sims_1} balances two competing goals: minimizing the
linear-quadratic control term (i.e., $u^\top R_1 u + R_2^\top u$), whose 
optimizer is $-\frac{1}{2}R_1^{-1}R_2$, and regulating the system output to 
the origin (through $\| y\|^2$), robustly against any arbitrary $w_t$. 
The parameter $\gamma$ controls the tradeoff between the two 
objectives. We consider the following instance of \eqref{eq:plant}:
\begin{align}
x_{t+1} &= A x_t + B u_t + E w_t, & 
y_t &= C x_t + D w_t,
\label{system}
\end{align}
The steady-state map \eqref{eq:ssMap} of this plant is given by 
$y_{\mathrm{ss}}(u,w)
=
Gu+Hw,$
where
$G=C(I-A)^{-1}B,
H=C(I-A)^{-1}E+D,$
and terminal cost
$\Psi(u,y)
=
u^\top R_1 u
+
R_2^\top u
+
\gamma\|y\|^2.$

Although we present our simulations in a unified framework following the feedback 
optimization formulation of \Cref{sec:feedbackOptimization_algorithms}, we 
evaluate the proposed methods under both exact and inexact oracle models. 
We consider an exact oracle \eqref{eq:oracle_feedback_optimization_exact} given by
\begin{align}\label{eq:oracle_feedback_optimization_exact_sims}
\Phi_t(u)
&=
\Psi \big(u, y_{\mathrm{ss}}(u,w_t)\big) 
=
u^\top R_1 u
+
R_2^\top u
+
\gamma\|Gu+Hw_t\|^2.
\end{align}
while the inexact oracle~\eqref{eq:oracle_feedback_optimization_v2} is given by:
\begin{align}\label{eq:oracle_feedback_optimization_v2_sims}
\tilde{\Phi}_t(u)
&=
\Psi\big(u, h(f(x_{t}, u, w_{t}), w_t)\big) \\
&=
u^\top R_1 u
+
R_2^\top u \notag +
\gamma
\left\|
CAx_{t}
+
CBu
+
CEw_{t}
+
Dw_t
\right\|^2.
\end{align}

For our numerics, we use $p=q=r=5$, $n=10$. We parameterize $R_1$ as
$R_1 = R_3^\top R_3,$ and draw the entries of $R_2$ and $R_3$ independently from 
the standard uniform distribution. We set $\gamma=1$. The matrices $A,B,C,D,E$ 
are also generated randomly from the uniform distribution, with $\|A\|=0.05$. 
We let $w_t$ be a realization of 
$w_t \sim \mathcal{N}(w^\star,\frac{\sigma^2}{1+t} I_r),$ $w^\star \in \real^r$  $\sigma^2 >0,$ modeling a disturbance sequence with decaying variance, corresponding to a progressively more accurate oracle. 
For the numerical experiments, the disturbance sequence is generated once at the beginning and then kept fixed throughout all runs; consequently, conditionally on this realization, $\{\Phi_t\}_{t \geq 0}$
can be regarded as a deterministic sequence of objective functions.
Unless otherwise specified, we use $\sigma=1$.
Notice that, with this model, 
$\mathbb{E}\left\|w_{t+1}-w_t\right\|=O\left(\frac{\sigma}{\sqrt{t+1}}\right)$, corresponding 
to $\sum_{t=0}^{T-1} \mathbb{E}\left\|w_{t+1}-w_t\right\|=O(\sigma \sqrt{T}).$
Unless specified otherwise, we use the choice of probing 
$\delta_t = \frac{1}{\sqrt{t+1}}.$


\subsection{Validation of the analytical bounds}
\label{sec:validation_bounds}

Simulation results are reported in \Cref{fig:bound_and_noise} (left). 
The blue curve shows 
the evolution of $\|\nabla \Phi_t(u_t)\|$ for an algorithm with exact oracle 
(i.e., \Cref{alg:two_point_ds_tv} with oracle 
\eqref{eq:oracle_feedback_optimization_exact_sims}), while the orange curve 
corresponds to an algorithm with inexact oracle (i.e., \Cref{alg:two_point_ds_fo} with oracle
\eqref{eq:oracle_feedback_optimization_v2_sims}). 
The figure shows that the gradient error in both cases decays linearly with 
slope approximately $-1/2$ in log--log scale, consistently with the 
$\mathcal{O}(\log(T)/\sqrt{T})$ resolution predicted by 
\Cref{thm:two_point_diminishing} (cf. \eqref{eq:resolution_diminishing}).

The green curve illustrates a trajectory-based empirical counterpart
of the upper bound of 
\Cref{thm:two_point_tv_inexact}. Precisely, rearranging \eqref{eq:main-step} 
yields the pointwise estimate
\begin{align}\label{eq:pointwise_bound}
\mathbb{E}\!\left[\|\nabla \Phi_t(u_t)\|\right]
&\le
\sqrt{2\pi p}
\bigg(
\frac{
\mathbb{E}[\Phi_t(u_t)-\Phi_{t+2}(u_{t+2})]
}{\delta_t}
+
\frac{L_{\nabla \Phi}}{4}\delta_t
\notag \\
&\qquad+
\frac{
2d_t+d_{t+1}+2b_t +2b_{t+1}
}{\delta_t} \bigg).
\end{align}
The green curve of \Cref{fig:bound_and_noise} (left) illustrates the right-hand side of \eqref{eq:pointwise_bound}, where the quantities $d_t$ and $b_t$ are estimated 
empirically along the iterates as
\[
\hat d_t
\defeq
|\Phi_{t+1}(u_t)-\Phi_t(u_t)|,
\qquad
b_t^{\mathrm{fo}}(u_t)
\defeq
|\tilde{\Phi}_t(u_t)-\Phi_t(u_t)|.
\]
%
The simulations show that this trajectory-based surrogate accurately captures 
the behavior of the method, with the observed performance remaining 
consistently below it. We stress that, since $\hat d_t$ evaluates the temporal 
drift along the trajectory rather than as a supremum over $u$, the green curve 
is an empirical surrogate for the right-hand side of 
\eqref{eq:pointwise_bound}, and not a certified upper bound. 

Asymptotically, the green curve in \Cref{fig:bound_and_noise} (left) reaches a plateau, 
which can be attributed to the drift-related contribution 
$\frac{
2d_t+d_{t+1}+2b_t +2b_{t+1}
}{\delta_t}$ 
in \eqref{eq:pointwise_bound}. This behavior can be interpreted in light of
\Cref{cor:two_point_fo_diminishing}: since 
$\sum_{t=0}^{T-1}\mathbb{E}\!\left[\|w_{t+1}-w_t\|\right]
=
O(\sigma \sqrt{T}),$
the term
$\frac{1}{\sqrt{T}}
\left(
\sum_{t=0}^{T-1}\|w_{t+1}-w_t\|
+
\sum_{t=0}^{T-1}\mathbb{E}[\mu_t]
\right)$ does not converge to zero; consequently, 
the method converges only to a neighborhood of a critical point, as predicted 
by \Cref{cor:two_point_fo_diminishing}.


\begin{figure}[t!]
    \centering

    \begin{minipage}[t]{0.53\textwidth}
        \centering
        \includegraphics[width=\textwidth]{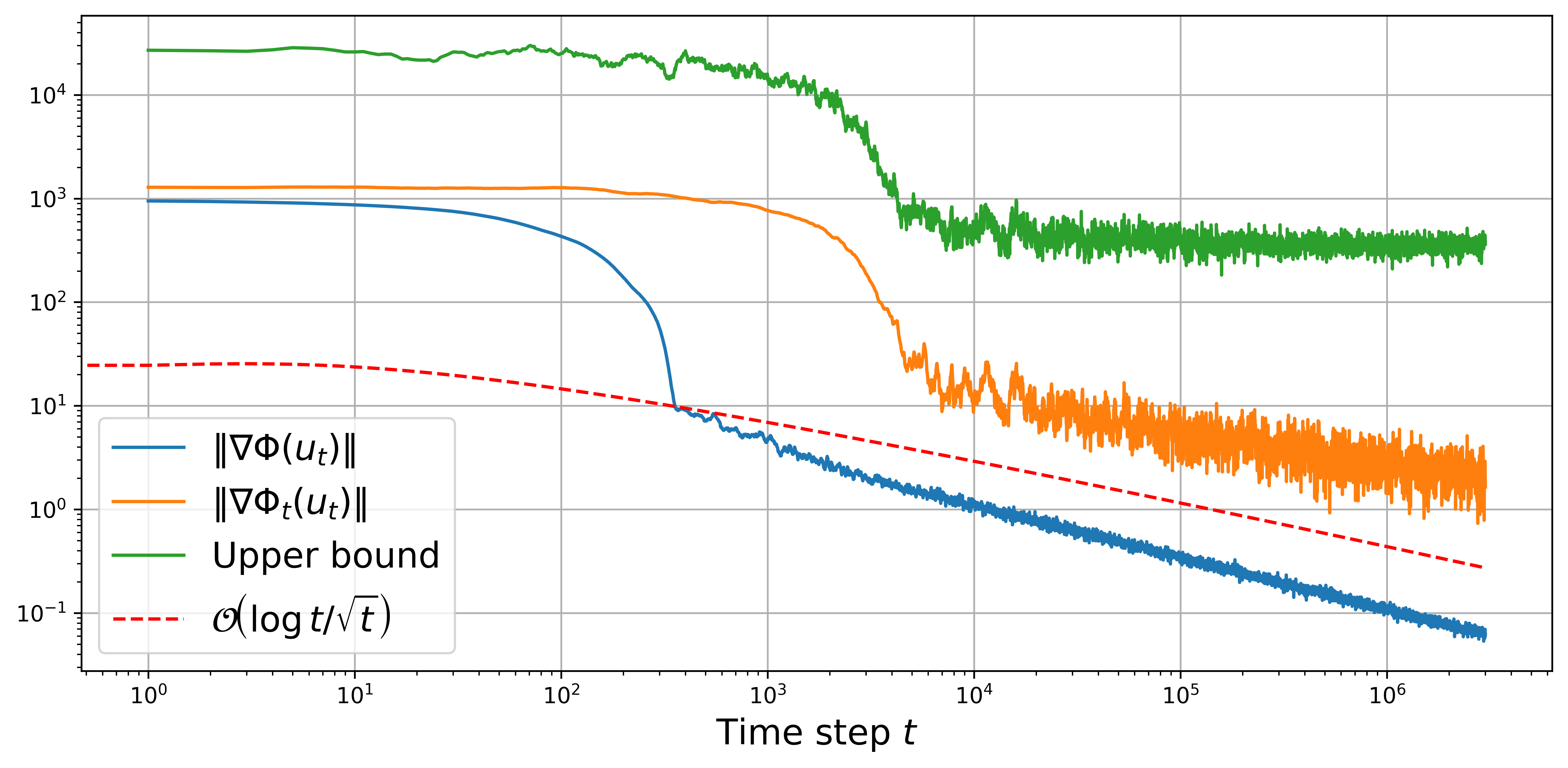}
    \end{minipage}
    \hfill
    \begin{minipage}[t]{0.46\textwidth}
        \centering
        \includegraphics[width=\textwidth]{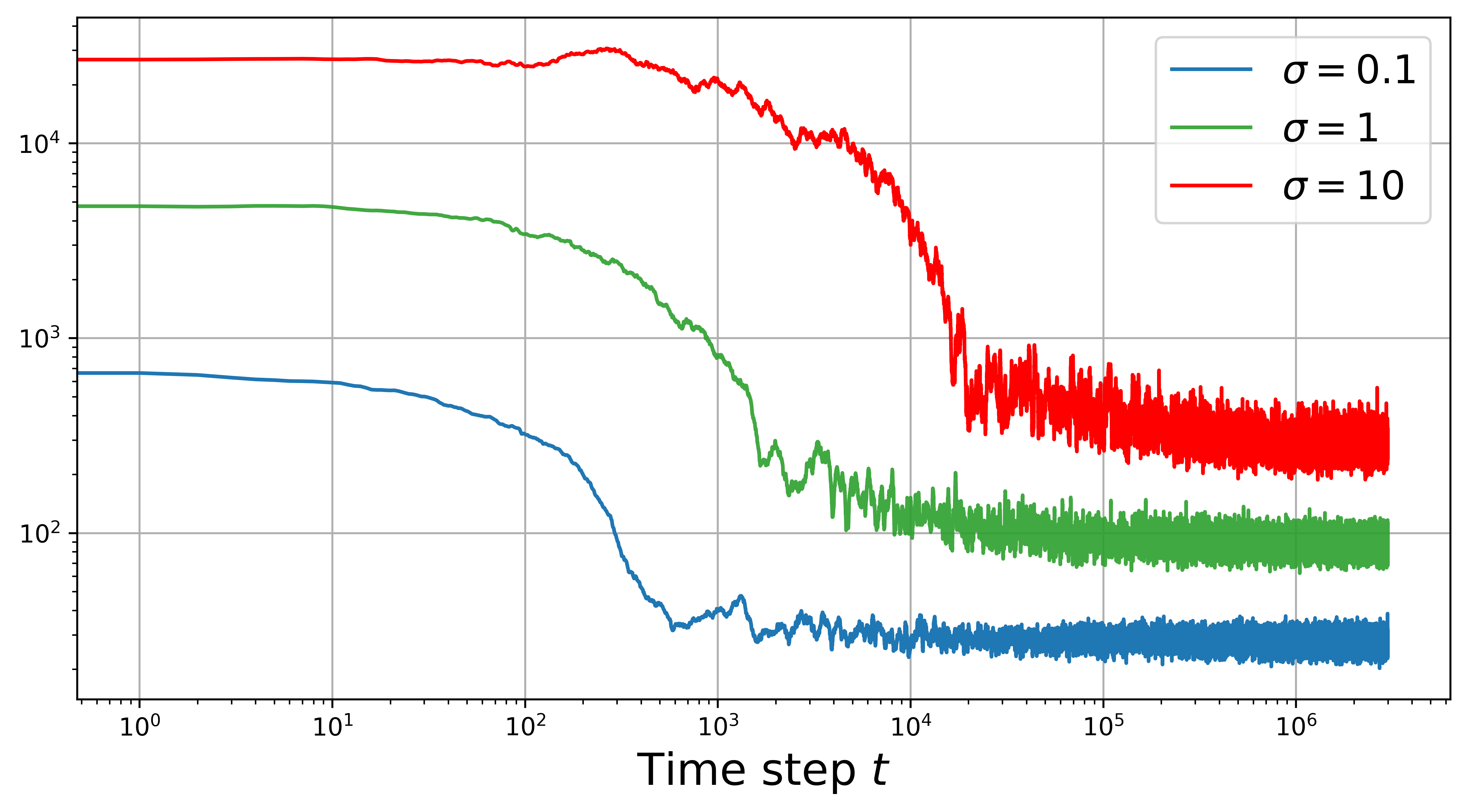}
    \end{minipage}

    \caption{Left: comparison of
    the exact- and inexact-oracle implementations with the trajectory-based
    empirical surrogate \eqref{eq:pointwise_bound} (labelled ``upper bound'' in
    the legend). Right: evolution of the surrogate
    \eqref{eq:pointwise_bound} for
    $\sigma\in\{0.1,\,1,\,10\}$. The plateau level increases with $\sigma$,
    consistently with
    \eqref{eq:low_bound_eps_two_point_fo_diminishing}, according to which the
    achievable resolution scales proportionally to the noise level. All
    quantities are averaged over 10 realizations of $\{v_t\}$.}
    \label{fig:bound_and_noise}
\end{figure}

\subsection{Dependence on noise magnitude}
The asymptotic accuracy of the proposed method is illustrated in
\Cref{fig:bound_and_noise} (right), which depicts the evolution of the right-hand side of
\eqref{eq:pointwise_bound} for $\sigma\in\{0.1,\,1,\,10\}$.
All three curves exhibit the same qualitative behavior: an initial transient 
decay followed by a plateau whose level increases with $\sigma$. This behavior 
is consistent with \eqref{eq:low_bound_eps_two_point_fo_diminishing}: 
under  the considered disturbance model,
$\sum_{t=0}^{T-1}\mathbb{E}\!\left[\|w_{t+1}-w_t\|\right]
=
O(\sigma \sqrt{T}),$
so \eqref{eq:low_bound_eps_two_point_fo_diminishing} implies 
$\varepsilon \ge \kappa \sigma$ for some $\kappa>0$, showing that the 
achievable resolution scales proportionally to the noise level $\sigma$.

\subsection{Comparison between one-, two-, and three-point methods}


\begin{figure}[t!]
    \centering

    \begin{minipage}[t]{0.49\textwidth}
        \centering
        \includegraphics[width=\textwidth]{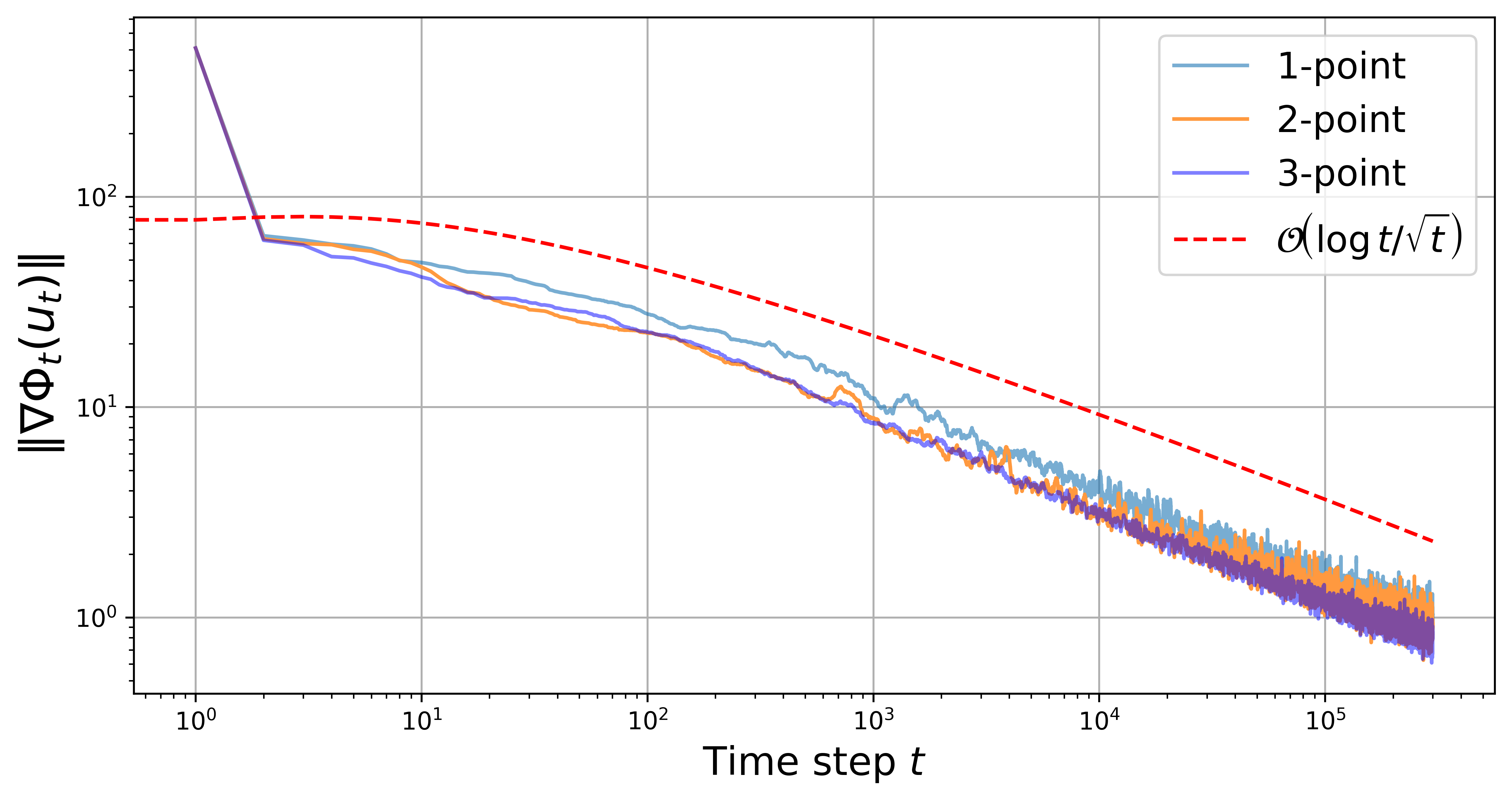}
    \end{minipage}
    \hfill
    \begin{minipage}[t]{0.49\textwidth}
        \centering
        \includegraphics[width=\textwidth]{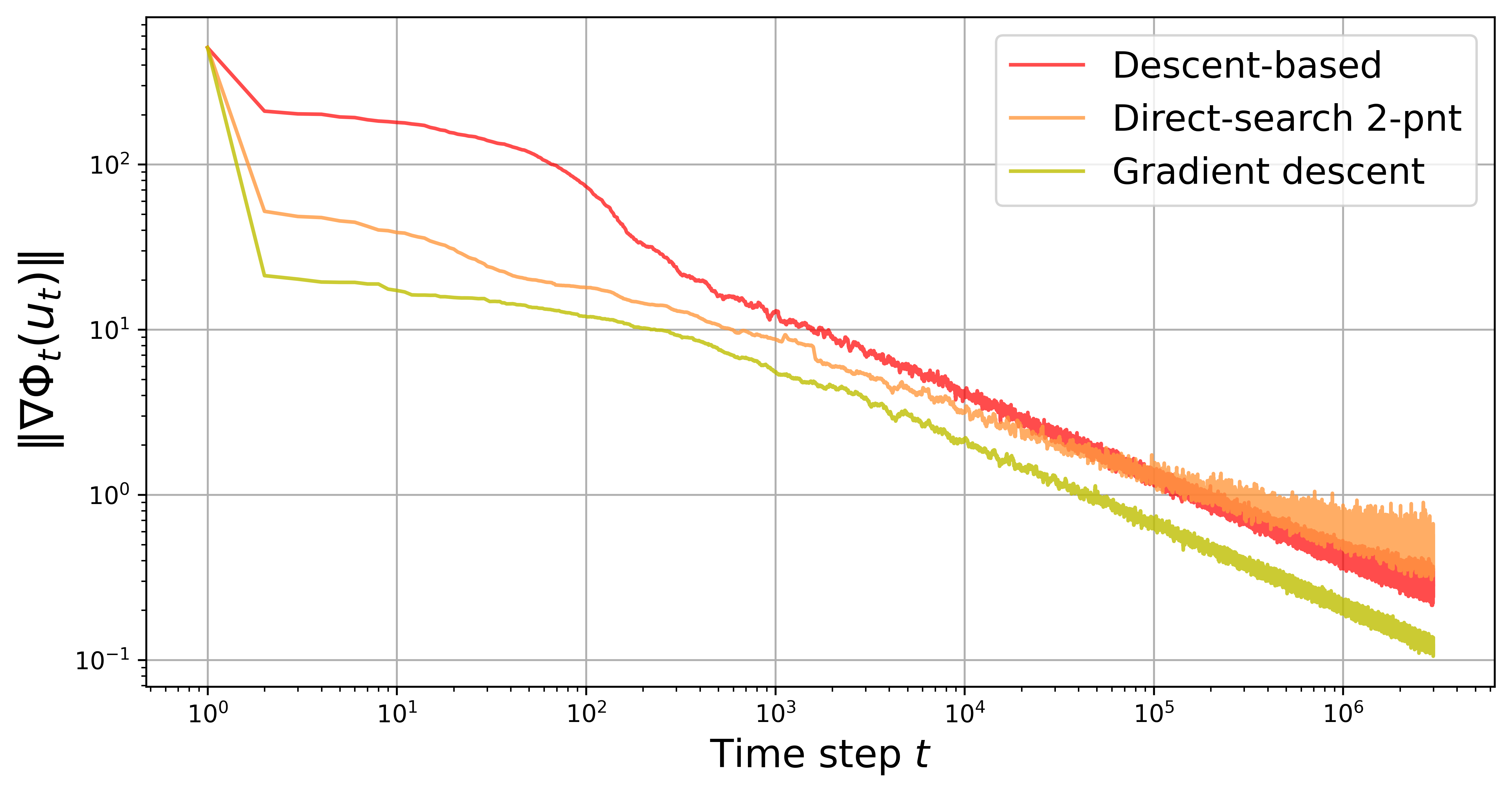}
    \end{minipage}

    \caption{Left:
    comparison of the one-point (\Cref{alg:one_point_ds_residual}), two-point
    (\Cref{alg:two_point_ds_tv}), and three-point
    (\Cref{alg:three_point_ds_tv}) direct-search methods under the inexact
    oracle~\eqref{eq:oracle_feedback_optimization_v2_sims}. All methods exhibit
    similar asymptotic behavior, with decay rates consistent with the predicted
    $\mathcal{O}(\log(T)/\sqrt{T})$ scaling. Right: comparison of the
    two-point direct-search method with online gradient
    descent~\cite{EH:16} and the two-point gradient-estimation
    method~\cite{YN-VS:17}, using the same inexact oracle. 
    }
    \label{fig:comparison_methods}
\end{figure}

\Cref{fig:comparison_methods} (left) compares the one-point 
(cf.~\Cref{alg:one_point_ds_residual}), two-point 
(cf.~\Cref{alg:two_point_ds_tv}), and three-point 
(cf.~\Cref{alg:three_point_ds_tv}) methods.
For a sharper comparison, in this simulation, all methods are 
implemented using the inexact oracle 
\eqref{eq:oracle_feedback_optimization_v2_sims}.
All three methods exhibit qualitatively similar behavior: after a comparable 
transient phase, the curves decay in agreement with the 
$\mathcal{O}(\log(T)/\sqrt{T})$ rate predicted by 
\Cref{thm:two_point_diminishing} (cf.~Remark~\ref{rem:resolution_diminishing}).

\blue{The three-point method achieves the best 
asymptotic accuracy among the three methods for this setting (a particular cost function, norm of $A$, noise level), suggesting that the additional 
exploratory perturbation improves the quality of the descent direction in 
the proximity of critical points.}

\subsection{Comparison with the state-of-the-art}
\Cref{fig:comparison_methods} (right) compares \Cref{alg:two_point_ds_tv} with online 
gradient descent~\cite{EH:16} (orange), implemented with the optimal stepsize 
$\eta=1/L$, and the two-point gradient-estimation method of~\cite{YN-VS:17} 
(blue). For a sharper comparison, all methods are implemented using the inexact 
oracle~\eqref{eq:oracle_feedback_optimization_v2_sims}. 
We stress that online gradient descent uses additional (first-order) 
knowledge relative to \Cref{alg:two_point_ds_tv}, which instead relies 
exclusively on zeroth-order information. 
On the other hand, the descent-based method of~\cite{YN-VS:17} 
uses zeroth-order information, but seeks to replicate the behavior of online 
gradient descent by producing gradient estimates from zeroth-order information. Furthermore, in contrast with the proposed method, this descent-based method can be unstable and requires stepsizes to be tuned in practice. 

The simulation shows that the proposed direct-search method achieves transient 
performance comparable to that of first-order methods. In contrast, the 
randomized schemes exhibit a higher asymptotic variance, attributable to the 
persistent random exploration noise used to identify descent directions. 
Overall, the results demonstrate that the proposed method attains performance 
comparable both to existing zeroth-order approaches and to first-order methods 
with access to exact gradients, highlighting the potential of the approach.


\subsection{Dependence on the problem dimension}
\label{sec:dimension_dependence}

\Cref{fig:dim_p} reports the number of iterations required to reach an
$\varepsilon$-stationary point as the decision dimension $p$ grows, for several
stationarity tolerances. The observed growth is approximately linear in $p$, in
agreement with the scaling
$T=\mathcal{O}(p\varepsilon^{-2}\log^2(L_{\nabla \Phi}\sqrt{p}/\varepsilon))$
predicted by \Cref{thm:two_point_diminishing}.

\begin{figure}[t!]
    \centering
    \includegraphics[width=0.75\columnwidth]{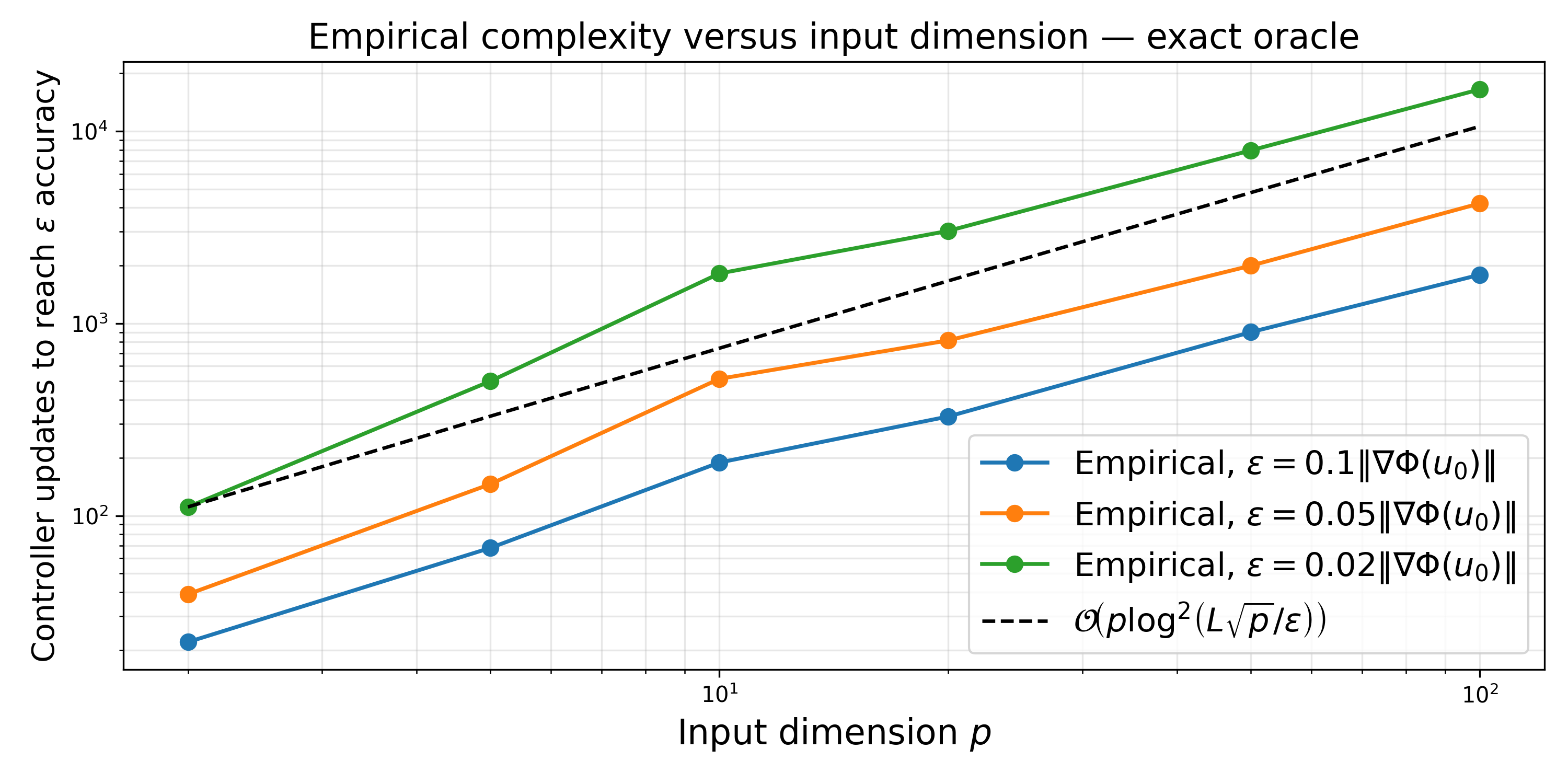}
    \caption{
    Iteration complexity of \Cref{alg:two_point_ds_tv}
    as a function of the input dimension \(p\) for different stationarity
    tolerances \(\varepsilon\).
    The figure illustrates the scaling predicted by
    \Cref{thm:two_point_diminishing}, namely
    \(\displaystyle
    T =
    \mathcal{O}\!\left(
    \frac{p}{\varepsilon^2}
    \log^2\!\left(
    \frac{L_{\nabla\Phi}\sqrt{p}}{\varepsilon}
    \right)
    \right)
    \),
    showing an approximately linear growth of the number of plant iterations
    with respect to \(p\), up to logarithmic factors.
    Smaller values of \(\varepsilon\) require more iterations to reach an
    \(\varepsilon\)-stationary point.
    See \Cref{sec:simulation_results} for the numerical setup.
    }
    \label{fig:dim_p}
\end{figure}

\section{Conclusions}
\label{sec:conclusions}

This paper developed a direct-search framework for nonconvex time-varying optimization under inexact zeroth-order information. Unlike existing direct-search methods, which are largely confined to static settings, the proposed approach explicitly accounts for temporal variability and oracle inaccuracies. We established non-asymptotic stationarity guarantees and iteration-complexity bounds under both constant and diminishing probing ratios, recovering classical zeroth-order rates up to logarithmic factors. We also showed how the framework applies to feedback optimization and online equilibrium selection in dynamical systems, yielding guarantees that explicitly quantify the effects of temporal drift and plant-induced oracle errors.
Promising directions of future work include analyzing one-point or three-point direct-search 
methods, constrained and distributed extensions, accelerated schemes, and the 
derivation of lower complexity bounds for time-varying zeroth-order optimization.

\section*{Acknowledgments}
The authors thank Geovani N.~Grapiglia for insightful discussions and valuable
input during the development of this work.


\appendix

\section{Auxiliary results}

\subsection{Sums of diminishing probing ratios}
\label{app:stepsize_sums}

\begin{lemma}[Diminishing probing ratio sums]
\label{lem:diminishing_stepsize_sums}
Let $T \ge 2$ be even and consider the probing ratios
\[
\delta_t = \frac{1}{\sqrt{t+1}},
\qquad t = 0,2,\dots,T-2.
\]
Then, the following bounds hold:
\begin{align}
\sum_{t \in \evenset{T}}\delta_t
&\ge
\frac{\sqrt{T}}{\sqrt{2}+1},
\label{eq:sum_delta_lower}
\\
\sum_{t \in \evenset{T}}\delta_t^2
&\le
1 + \frac{1}{2}\log(T).
\label{eq:sum_delta_sq_upper}
\end{align}
\end{lemma}

\begin{proof}
Let $t=2k$, with $k=0,1,\dots,\frac{T}{2}-1$. Then
\[
\delta_{2k} = \frac{1}{\sqrt{2k+1}}.
\]

For the first bound,
\begin{align*}
\sum_{t \in \evenset{T}}\delta_t
&=
\sum_{k=0}^{T/2-1} \frac{1}{\sqrt{2k+1}}
\;\ge\;
\int_{0}^{T/2} \frac{1}{\sqrt{2x+1}}\,dx \\
&=
\left[\frac{\sqrt{2x+1}}{1}\right]_{0}^{T/2}
=
\sqrt{T+1}-1
\;\ge\;
\frac{\sqrt{T}}{\sqrt{2}+1},
\end{align*}
where the last inequality holds for all $T \ge 2$.

For the second bound,
\begin{align*}
\sum_{t \in \evenset{T}}\delta_t^2
&=
\sum_{k=0}^{T/2-1} \frac{1}{2k+1}
\;\le\;
1 + \int_{1}^{T/2} \frac{1}{2x}\,dx \\
&=
1 + \frac{1}{2}\log\!\left(\frac{T}{2}\right)
\;\le\;
1 + \frac{1}{2}\log(T).
\end{align*}
\end{proof}

\subsection{Gaussian one-sided moments}
\label{app:gaussian}

\begin{lemma}[Gaussian one-sided moments]
\label{lem:gaussian_one_sided}
\\
Let $v \sim \mathcal{N}(0,\frac{1}{p}I_p)$. Then,
$\mathbb{E}[\|v\|^2]=1.$
Moreover, for any vector \(g\in\mathbb R^p\setminus\{0\}\), the following identities hold:
\begin{align}
    \label{eq:descent_product}
\mathbb{E}\!\left[(-\langle g,v\rangle)_+\right]
&=
\frac{1}{\sqrt{2\pi p}}\|g\|,
\\
\label{eq:descent_norm}
\mathbb{E}\!\left[\|v\|^2 \mathbf{1}_{\{\langle g,v\rangle<0\}}\right]
&=
\frac{1}{2}.
\end{align}
\end{lemma}


\begin{proof}
Since $v \sim \mathcal{N}\!\left(0,\frac{1}{p}I_p\right)$, its components satisfy 
$v_i \sim \mathcal{N}\!\left(0,\frac{1}{p}\right)$ for $i=1,\ldots,p$. Hence,
\[
\mathbb{E}\!\left[\|v\|^2\right]
=
\mathbb{E}\!\left[\sum_{i=1}^p v_i^2\right]
=
\sum_{i=1}^p \mathbb{E}\!\left[v_i^2\right]
=
\sum_{i=1}^p \mathrm{Var}(v_i)
=
1.
\]

We next prove~\eqref{eq:descent_product}. Suppose 
$Z \sim \mathcal{N}(0,\sigma^2)$. Then,
\begin{align*}
\mathbb{E}[(-Z)_+]
&=
\frac{1}{\sqrt{2\pi\sigma^2}}
\int_{-\infty}^{0} (-x)\, e^{-\frac{x^2}{2\sigma^2}} \, dx =
\frac{1}{\sqrt{2\pi\sigma^2}}
\int_{0}^{+\infty} y\, e^{-\frac{y^2}{2\sigma^2}} \, dy
=
\frac{\sigma}{\sqrt{2\pi}},
\end{align*}
where the second equality follows by letting $y=-x$.
In our case,
$\langle g,v\rangle
=
\sum_{i=1}^p g_i v_i
\sim
\mathcal{N}\!\left(0,\frac{1}{p}\|g\|^2\right),$
and therefore $\sigma^2 = \|g\|^2/p$. Substituting into the previous expression 
yields
$\mathbb{E}\!\left[(-\langle g,v\rangle)_+\right]
=
\frac{1}{\sqrt{2\pi p}}\|g\|,$
which proves~\eqref{eq:descent_product}.

Finally, we prove~\eqref{eq:descent_norm}. Since $v$ is symmetric about the origin and 
$\|v\|^2$ is an even function, we have
\[
\mathbb{E}\!\left[
\|v\|^2 \mathbf{1}_{\{\langle g,v\rangle<0\}}
\right]
=
\mathbb{E}\!\left[
\|v\|^2 \mathbf{1}_{\{\langle g,v\rangle>0\}}
\right].
\]
Since $g\neq 0$, then $\mathbb{P}(\langle g,v\rangle=0)=0$, and thus
\[
\mathbf{1}_{\{\langle g,v\rangle<0\}}
+
\mathbf{1}_{\{\langle g,v\rangle>0\}}
=
1
\quad \text{a.s.}
\]
It follows that
\[
2\,\mathbb{E}\!\left[
\|v\|^2 \mathbf{1}_{\{\langle g,v\rangle<0\}}
\right]
=
\mathbb{E}\!\left[\|v\|^2\right]
=
1,
\]
which implies~\eqref{eq:descent_norm}.
\end{proof}

\Cref{lem:gaussian_one_sided} quantifies the expected improvement obtained by 
probing the function along a random Gaussian direction. 
\Cref{eq:descent_product} shows that the expected one-sided projection of $v$ 
against $g$ scales linearly with $\|g\|$, providing a descent signal of order 
$1/\sqrt{p}$. Noting that $\mathbf{1}_{\{\langle g_t,v_t\rangle<0\}}$ indicates 
whether $v_t$ is a descent direction for $g_t$, \eqref{eq:descent_norm} states 
that the squared norm of $v$, restricted to the event that $v$ is a descent 
direction for $g$, has expectation $\frac{1}{2}$. Equivalently, since that event 
has probability $\frac{1}{2}$, the conditional expectation is
$\mathbb{E}\!\left[\|v\|^2 \mid \langle g,v\rangle<0\right]=1$.




\bibliographystyle{siamplain-nodoi}
\bibliography{BIB/alias,BIB/full_GB,BIB/GB,BIB/references}

\end{document}